\documentclass[11pt]{amsart}

\RequirePackage{amsthm,amsmath,amsfonts,amssymb}
\RequirePackage[numbers,sort&compress]{natbib}
\RequirePackage[colorlinks,citecolor=blue,urlcolor=blue]{hyperref}
\RequirePackage{graphicx}

\usepackage{mathtools}
\usepackage{enumerate}
\usepackage{color}
\usepackage{framed} 
\usepackage{ulem} 
\usepackage{bbm}

\allowdisplaybreaks[4]

\theoremstyle{plain}
\newtheorem{theo}{Theorem}[section]
\newtheorem{sat}[theo]{Proposition}
\newtheorem{lem}[theo]{Lemma}
\newtheorem{prop}[theo]{Proposition}

\newtheorem{korr}[theo]{Corollary}

\newtheorem{de}[theo]{Definition}

\newtheorem{example}[theo]{Example}
\newtheorem{remark}[theo]{Remark}
\newtheorem{remarks}[theo]{Remarks}

\newcommand{\BQN}{\begin{eqnarray}}
\newcommand{\EQN}{\end{eqnarray}}
\newcommand{\BQNY}{\begin{eqnarray*}}
\newcommand{\EQNY}{\end{eqnarray*}}
\newcommand{\BS}{\begin{sat}}
\newcommand{\ES}{\end{sat}}
\newcommand{\BT}{\begin{theo}}
\newcommand{\ET}{\end{theo}}
\newcommand{\BK}{\begin{korr}}
\newcommand{\EK}{\end{korr}}
\newcommand{\BD}{\begin{de}}
\newcommand{\ED}{\end{de}}
\newcommand{\BIT}{\begin{itemize}}
\newcommand{\EIT}{\end{itemize}}
\newcommand{\BDI}{\begin{description}}
\newcommand{\EDI}{\end{description}}
\newcommand{\BRM}{\begin{remarks}}
\newcommand{\ERM}{\end{remarks}}
\newcommand{\BEL}{\begin{lem}}
\newcommand{\EEL}{\end{lem}}

\newcommand{\nelem}[1]{{Lemma \ref{#1}}}

\newcommand{\neprop}[1]{{Proposition \ref{#1}}}
\newcommand{\netheo}[1]{{Theorem \ref{#1}}}

\newcommand{\abs}[1]{\left\lvert #1 \right\rvert}

\newcommand{\E}[1]{\mathbb{E}\left\{ #1\right\}}

\newcommand{\pk}[1]{\mathbb{P} \left\{ #1 \right \} }
\newcommand{\R}{\mathbb{R}}
\newcommand{\N}{\mathbb{N}}

\newcommand{\ldot}{,\ldots,}
\newcommand{\limit}[1]{\lim_{#1 \to \infty}}

\newcommand{\COM}[1]{}

\newcommand{\QED}{\hfill $\Box$}
\newcommand{\kb}[1]{\boldsymbol{#1}}
\newcommand{\vk}[1]{\kb{#1}}

\def\td{\text{\rm d}}

\DeclareMathOperator{\Var}{Var}
\DeclareMathOperator{\Cov}{Cov}
\DeclareMathOperator{\rank}{rank}
\DeclareMathOperator{\sign}{sign}

\definecolor{c20}{rgb}{0.,0.7,0.}
\definecolor{c30}{rgb}{0.,0.,1.}
\definecolor{c40}{rgb}{1,0.1,0.7}
\definecolor{c50}{rgb}{1,0,0}
\definecolor{c60}{rgb}{1,0.9,0.1}

\def\kd#1{{\textcolor{black}{#1}}}
\def\kk#1{{\textcolor{black}{#1}}}

\def\kkk#1{{\textcolor{black}{#1}}}

\begin{document}

\title{Extremes of Brownian Decision Trees}

\author{Krzysztof D\c{e}bicki}
\address{Krzysztof D\c{e}bicki, Mathematical Institute, University of Wroc\l aw,
	pl. Grunwaldzki 2/4, 50-384 Wroc\l aw,
	Poland}
\email{Krzysztof.Debicki@math.uni.wroc.pl}

\author{Pavel Ievlev}
\address{Pavel Ievlev, Department of Actuarial Science, 
	University of Lausanne,\\
	UNIL-Dorigny, 1015 Lausanne, Switzerland
}
\email{Pavel.Ievlev@unil.ch}

\author{Nikolai Kriukov}
\address{Nikolai Kriukov, Department of Actuarial Science, 
	University of Lausanne,\\
	UNIL-Dorigny, 1015 Lausanne, Switzerland
}
\email{Nikolai.Kriukov@unil.ch}

\bigskip

\date{\today}
\maketitle

{\bf Abstract:} We consider a Brownian motion with linear drift that splits at
fixed time points into a fixed number of branches, which may depend on the
branching point. For this process, which we shall refer to as the Brownian decision tree, we investigate the exact asymptotics of high exceedance probabilities in finite time horizon, including: the probability that at least one branch exceeds some high threshold, the probability that the largest distance between branches gets large
and the probability that all branches simultaneously exceed some high barrier.
Additionally, we find the asymptotics for the probability that all branches of at least one of $M$ independent Brownian decision trees exceed a high threshold.



\section{Introduction}
\kk{We} investigate the exact asymptotics of high exceedance
probabilities for the process, which we shall refer to as \textit{the Brownian decision
	tree}. This process is a close relative of the standard
\kk{{\it branching Brownian motion (BBm)}} and it can be informally described as follows: at time
\( t = 0 \) a Brownian motion \( B ( t ) \) sets off from zero and runs freely
until a non-random time \( \tau_1 > 0 \), at which it splits into \( N_1 \geq 1 \)
\kd{conditionally on the common past independent Brownian motions}
\begin{equation*}
	B ( t )
	\xmapsto[\text{at } \tau_1]{\text{branching}}
	\vk{B}_1 ( t )
	=
	\begin{pmatrix}
		B ( \tau_1 ) \\
		\vdots \\
		B ( \tau_1 )
	\end{pmatrix}
	+\vk{B}^{*}_1 ( t - \tau_1 ),
\end{equation*}
where $\vk{B}^*_1$ is an \( \R^{N_1} \)-valued Brownian motion. The resulting
vector-valued process again runs freely up to some time point \( \tau_2 > \tau_1 \),
where \textit{each of its components} \kk{splits} again into \( N_2 \geq 1 \) particles
\begin{equation*}
	\vk{B}_1 ( t )
	\xmapsto[\text{at } \tau_2]{\text{branching}}
	\vk{B}_2 ( t )
	=
	\begin{pmatrix}
		\vk{B}_1 ( \tau_2 ) \\
		\vdots \\
		\vk{B}_1 ( \tau_2 )
	\end{pmatrix}
	+\vk{B}_2^{*} ( t - \tau_2 )
\end{equation*}
and the construction recursively repeats.

There are two differences between this and the classical BBm model as presented,
for example, in the seminal paper by Bramson~\cite{bramson1978maximal}. \kk{Firstly}, the branching times are
non-random, whereas in the standard model the distances between them are
exponentially distributed. Secondly, all branches (that is, the components of
the vector-valued process described above) undergo splitting into the same
number of offsprings and at the same time (in the classical
BBm model each branch has its own branching clock). This, along with the usual
description of the classical BBm model as a process indexed by a tree, suggests
the name \textit{Brownian decision trees}, where the word ``decision'' refers to the
specific type of trees branching at the same points and into the same amount of
branches.

The main findings of the paper are collected in Section \ref{sec:main-results},
which we begin with two preliminary results.
In Section \ref{s.1} we consider the exact asymptotics of the probability that at least one branch of
the Brownian decision tree with drift exceeds $u$, as $u\to\infty$, that is,
\begin{eqnarray}
\pk{\exists \, t\in[0,T], \, \exists \, \gamma\in\Gamma \colon B_\gamma(t)-ct>u},\label{p1}
\end{eqnarray}
where \( \Gamma =\{0,\ldots,n-1\}\) is the set of indices of the decision tree branches and
\( c \in \mathbb{R} \) is a \kk{deterministic constant.
In Theorem \ref{single_ruin} we show that (\ref{p1}) is asymptotically equal to the product of
the number of branches at time $T$ and the probability that a single Brownian motion with drift $c$
crosses level $u$ in time interval $[0,T]$.}

A similar approach to the above can be applied for the exact asymptotics of the largest distance between the
branches
\begin{eqnarray}
\pk{\exists \, t\in[0,T], \exists \, \gamma_1,\gamma_2\in\Gamma \colon B_{\gamma_1}(t)-B_{\gamma_2}(t)>u},\label{p2}
\end{eqnarray}
as $u\to\infty$, which is derived in Theorem \ref{diameter_1}.

In Section \ref{s.3} we focus on
the
probability
\begin{equation}\label{eq:1}
	\mathbb{P} \left\{
		\exists \, t \in [ 0, T ] \ \forall \, \gamma \in \Gamma \colon
		B_{\gamma} ( t ) - c t  > u
	\right\}
\end{equation}
that \kd{for some $t\in[0,T]$ all branches exceed threshold $u$}.
This problem is much more complicated and needs a more subtle approach than used in the analysis
of (\ref{p1}) and (\ref{p2}).
\kkk{In order to get the exact asymptotics of (\ref{eq:1})
we develop the technique introduced in~\cite{MR4127347} for extremes of centered vector-valued Gaussian processes to the branching model considered in this contribution.}
The main result of this section is displayed in Theorem \ref{main}, which is
supplemented by an extension of Korshunov-Wang inequality~\cite{korshunov2020tail,MR4376582,MR4467243} (see
\neprop{Korsh}) where a tight upper bound for  (\ref{eq:1}), that is valid for all $u>0$, is derived.
Complementary to the above findings,
we investigate the asymptotics of (\ref{eq:1}) for
a version of Brownian decision tree
with random numbers of offsprings (Corollary \ref{random_N}) and
the
limiting distribution of the corresponding conditional \kd{exceedance} times
(Corollary \ref{ruintime}).

In Section~\ref{sec:forest} we  present asymptotic results 
for the process which we
call \textit{Brownian decision forest}, \kd{that consists of}  
a family of \( M \) independent Brownian decision trees \( \vk{B}_{\Gamma_i} \),
which grow from different points \( x_i \in \mathbb{R} \). 
For this purpose we introduce a partial order on the set of tuples
\( ( \vk{\tau}, \vk{N}, c, x ) \), determining the trees of a forest, such that the
trees which are high in this order are more likely to exceed the high barrier.
In Theorem~\ref{forest_ruin} we use this order to find the exact asymptotics of the
probability that all branches of at least one tree in a forest exceed some
high barrier simultaneously.

We note that in the context of the classical branching Brownian motion most of the asymptotical
results focus either on the number of particles or the distribution of
the highest branch, which in our setup would be \( \max_{\gamma \in \Gamma} B_{\gamma} ( t ) \).
In both cases the limiting parameter is \( t \) approaching infinity. We refer
the reader to the classical papers~\cite{bramson1978maximal,lalley1987,chauvin1988kpp,arguin2013extremal,mallein2015maximal}.
\kkk{The two types of questions mentioned above have been investigated extensively in
the last decades for various versions of the classical BBm model. Among these
versions, the most popular is the Kesten's BBm with absorption (see the original
paper~\cite{kesten1978branching}). Other models include BBm in random~\cite{hou2022invariance,vcerny2022tightness} and periodic~\cite{D4}
environments, spatial selection such as in the Brownian bees model~\cite{siboni2021fluctuations,berestycki2022brownian,meerson2021persistent},
spatially-inhomogenous branching rates~\cite{Bocharov2013BranchingBM}, self-repulsive BBm~\cite{Bovier2021BranchingWithSelfRepulsion} and many
others.}

As it turns out,  the exact asymptotics of the counterpart of (\ref{eq:1}) for the classical BBm
is relatively simpler to obtain
than for the analyzed in this contribution Brownian decision trees.
Namely, it suffices to observe that the most probable scenario in this case
is that the BBm process does not manage to branch even once before hitting the high
barrier. Hence, the
probability~\eqref{eq:1} is asymptotically equivalent to the probability that
the first branching event happens after \( T \) times the one-dimensional ruin
probability of the Brownian motion. We present here the precise result, the
proof of which the reader can find in the Appendix.

\begin{prop}\label{classical}
	Let \( \vk{B} ( t ) \), \( t \geq 0 \) be the classical Branching Brownian
	motion as described above, with \( \tau \sim \operatorname{Exp} ( 1 ) \). Then,
	for any \( c \in \mathbb{R} \) we have
	\begin{equation*}
		\pk{ \exists \, t \in [ 0, T ] \ \forall \, \gamma \in \Gamma \colon B_{\gamma} ( t ) - c t > u }
		\sim
		e^{-T} \times
		\sqrt{\frac{2 T}{\pi}} \exp \left( -\frac{( u + c T )^2}{2T} \right).
	\end{equation*}
\end{prop}

The paper is organized as follows. We begin with a construction of the Brownian
decision tree process and establish some necessary notation and basic properties
in Section~\ref{sec:definitions}. The main results are presented in Section~\ref{sec:main-results}.
Most of the technical proofs are relegated to Sections~\ref{sec:proofs}
and~\ref{sec:auxiliary-proofs}.

\section{Definitions \kk{and basic properties of Brownian decision trees}}
\label{sec:definitions}

Let \( 0 < \tau_1 < \ldots < \tau_{\eta} < T \) be a finite collection of points, further
referred to as \textit{the branching points}, and a sequence of numbers
\( N_i \geq 1 \) with \( i = 1, \ldots, \eta \), interpreted as \textit{the numbers of
	branches generated at} \( \tau_i \). Denote by \( P_i \) the total number of
branches born up to time \( \tau_i \), that is,
\begin{equation*}
	P_i = \prod_{j = 1}^i N_j,
\end{equation*}
and define the set \( \Gamma \) indexing the branches on \( [0, T] \) by
\begin{equation*}
\Gamma = \{ 0, 1, \ldots, P_{\eta} - 1 \}.
\end{equation*}
Let $I_i$ be the $P_i\times P_i$ identity matrix, and $\vk 1_i=(1,\ldots,1)^\top\in\R^{P_i}$.
Denote by \( i ( t ) \) the number of branching points before \( t \)
\begin{equation*}
i ( t ) = \max \{ i \in \{ 1, \ldots, \eta \} \colon t > \tau_i  \}.
\end{equation*}
Clearly, each \( t \) belongs to the corresponding interval of the form
\( ( \tau_{i ( t )}, \tau_{i ( t ) + 1} ] \).

Next, take a collection of mutually independent standard Brownian motions
$\vk{B}_i^{ * } ( t )$, $t\geq 0$ in $\R^{P_i}$ indexed by \( i = 1, \ldots, \eta \) and a
one dimensional Brownian motion $ B_0^{ * } ( t )$, $t \geq 0$. Assume that all
these processes are mutually independent. Using the ingredients described above,
we construct a new process $\widetilde{\vk{B}}_{\Gamma} ( t )$, $t \in [0, T]$, which
we shall call \textit{the Brownian decision tree}, as follows: for
\( t \in ( \tau_i, \tau_{i + 1} ] \) set
\begin{equation*}
	\widetilde{\vk{B}}_{\Gamma} ( t )
	\coloneqq
	\begin{pmatrix}
		\widetilde{\vk{B}}_{\Gamma} ( \tau_i ) \\
		\vdots \\
		\widetilde{\vk{B}}_{\Gamma} ( \tau_i )
	\end{pmatrix}
	+\vk{B}_i^{ * } ( t - \tau_i )
	\quad \text{for} \quad
	t \in ( \tau_i, \tau_{i + 1} ]
\end{equation*}
and
\begin{equation*}
	\widetilde{\vk{B}}_{\Gamma} ( t )
	\coloneqq B_0^{ * } ( t )
	\quad \text{for} \quad t \in [0, \tau_1].
\end{equation*}
Note that \( ( \widetilde{\vk{B}}_{\Gamma} ( t ) )_{t \in ( \tau_i, \tau_{i + 1} ]} \)
belongs to \( C ( ( \tau_i, \tau_{i + 1} ], \mathbb{R}^{P_i} ) \).
We can extract the individual branch indexed by \( \gamma \in \Gamma \) by  taking
\begin{equation}\label{B_g}
	B_{\gamma} ( t )
	\coloneqq
	\big( \widetilde{\vk{B}}_{\Gamma} ( t ) \big)_{( \gamma \bmod P_{i ( t )} ) + 1}
\end{equation}
and
denote for two fixed branches $\gamma_1$ and $\gamma_2$ their {\it separation moment}
at which they diverge by
\begin{equation*}
	\kappa ( \gamma_1, \gamma_2 )
	\coloneqq
	\min \left\{ n \in \{ 1, \ldots, \eta \} \colon \gamma_1 \neq \gamma_2 \bmod P_n \right\}.
\end{equation*}
Clearly,
$B_{\gamma_1} ( t ) = B_{\gamma_2} ( t )$ for \( t \leq \tau_{\kappa ( \gamma_1, \gamma_2 )} \).

Next, we present \kd{some useful properties} of \kkk{Brownian decision trees}.

\begin{prop}\label{b_decomp}
	For $\gamma\in\Gamma$ the process $B_\gamma$ is a Brownian motion and \kkk{for  $t\in(\tau_1,T]$
	it} admits the following representation in terms of a collection of
	independent Brownian motions \( \{ \vk{B}^{ * }_j \}_{j = 0, \ldots, \eta} \):
	\begin{equation}
		\label{b_sum}
		B_\gamma(t)
		=B_0^*(\tau_1)
		+\sum_{j=1}^{i(t)-1} \big( \vk{B}^*_{j} \big)_{(\gamma \bmod P_{j})+1}(\tau_{j+1}-\tau_j)
		+\big( \vk{B}^*_{i(t)} \big)_{(\gamma \bmod P_{i(t)})+1}(t-\tau_{i(t)}).
	\end{equation}
	Moreover, for $\gamma_1\not=\gamma_2$ and $t_1, \, t_2\in[0,T]$ holds
\begin{equation}\label{cov_formula}
	\Cov(B_{\gamma_1}(t_1),B_{\gamma_2}(t_2))
	= \min \{ t_1, t_2, \tau_{\kappa(\gamma_1,\gamma_2)} \}.
\end{equation}
\end{prop}

The proof of \neprop{b_decomp} can be found in Appendix.

Sometimes it is more convenient to work with the process $\vk B_\Gamma$, which
contains all the branches simultaneously
\begin{equation}\label{B_G}
	\vk B_\Gamma(t)=(B_\gamma ( t ))_{ \gamma \in \Gamma } \in \R^{P_\eta}.
\end{equation}
The difference between this process and $\widetilde{\vk B}_{\Gamma}$ is that its
values belong to \( \mathbb{R}^{P_{\eta}} \) for each \( t \) instead of
\( \mathbb{R}^{P_{i ( t )}} \). The two processes are related as follows:
\begin{equation}
	\label{B_to_B}
	\widetilde{\vk B}_\Gamma(t)=( \vk B_\Gamma ( t ) )_{\gamma\in\{0,\ldots,P_{i(t)}-1\}}.
\end{equation}
Unlike \( \widetilde{\vk{B}}_{\Gamma} ( t ) \), the variance matrix of this process
is degenerate for all $t\in[0,\tau_{\eta}]$:
\begin{equation*}
	\det\left( \Var\left(\vk B_\Gamma(t)\right)\right)=0.
\end{equation*}

\begin{prop}\label{ind_increments}
	For  $0\leq t_1<t_2\leq T$, the random vector $\vk B_\Gamma(t_2)-\vk B_\Gamma(t_1)$ is
	independent of the process $\left.\vk B_\Gamma(t)\right|_{t\in[0,t_1]}$.
\end{prop}

Let $\Sigma(t)$ denote the covariance matrix of the random vector
$\widetilde{\vk B}_{\Gamma}(t)$:
\begin{equation*}
	\Sigma(t) \coloneqq
	\E{\widetilde{\vk B}_{\Gamma}(t) \, \widetilde{\vk B}^\top_{\Gamma}(t)},
	\qquad t \in [0, T].
\end{equation*}
\begin{prop}\label{Sigma_rec}
For  $t\in(\tau_1,T]$,
\begin{equation*}
	\Sigma(t)
	=
	\begin{pmatrix}
		\Sigma(\tau_{i(t)}) &\ldots & \Sigma(\tau_{i(t)}) \\
		\vdots & \ddots & \vdots
		\\ \Sigma(\tau_{i(t)}) &\ldots & \Sigma(\tau_{i(t)})
	\end{pmatrix}
	+(t-\tau_{i(t)})I_{i(t)},
\end{equation*}
where the first term matrix has $N_{i(t)}^2$ blocks, all equal to $\Sigma(\tau_{i(t)})$.
\end{prop}

\kk{In the next proposition we find the eigenvalues of $\Sigma(t)$.}

\begin{prop}\label{eigenvalues}
	For  $t\in[0,T]$, the eigenvalues of matrix $\Sigma(t)$ are given by
	\begin{equation}\label{mu_def}
		\mu_v(t)=(t-\tau_{i(t)})+\sum_{l=v+1}^{i(t)}(\tau_l-\tau_{l-1})\prod_{j=l}^{i(t)}N_j,
		\qquad v=0,1,\ldots,i(t).
	\end{equation}
	The multiplicity of $\mu_v(t)$ equals $P_v-P_{v-1}$, except for $\mu_0$, the
	multiplicity of which is $1$. Additionally, for any $t\in[0,T]$, the vector
	$\vk 1_{i(t)}$ is an eigenvector of $\Sigma(t)$ corresponding to $\mu_0(t)$.
\end{prop}

Using~\neprop{eigenvalues} we can obtain the following result.

\begin{prop}\label{w(t)_asympt} For  $c\in\R$ \kk{and $T>0$}
	\begin{align*}
		\pk{\vk B_\Gamma(T)-cT\vk 1_\eta>u\vk 1_{\eta}}\sim u^{-P_\eta}\frac{\mu_0^{P_\eta-1/2}(T)}{(2\pi)^{P_\eta/2}\prod_{v=1}^{\eta}\mu_v^{(P_v-P_{v-1})/2}(T)}\exp\left(-\frac{(u+cT)^2P_\eta}{2\mu_0(T)}\right)
	\end{align*}
	as $u\to\infty$, where $\mu_v(T)$ are defined in~\eqref{mu_def}.
\end{prop}

\begin{remark}
	The result obtained in \neprop{w(t)_asympt} still hold for $T=T_u$ with
	$T_u\to T$ as $u\to\infty$.
\end{remark}
The proofs of \neprop{eigenvalues} and \neprop{w(t)_asympt} are given in
Appendix.

\section{Main results}
\label{sec:main-results}
In this section we present the main findings of this contribution.
We begin with the asymptotic analysis of
the high exceedance probability of at least one branch of $\vk B_\Gamma$ with linear drift, then we proceed to
the asymptotics of the largest distance between the
branches. {In Section \ref{s.3}} we investigate the probability that all branches of $\vk B_\Gamma$ with linear drift
exceed high threshold and \kkk{then} extend it to the analysis of a finite collection of independent Brownian
decision trees. 

\subsection{High-exeedance of at least one branch}\label{s.1}
\kd{Consider} the probability that at least one branch of the
branching Brownian decision tree with linear trend exceeds some large threshold $u$. 
Clearly, for each $u>0$ and a standard Brownian motion $B(t), t\in[0,\infty)$
\[
\pk{\exists \, t\in[0,T], \, \exists \, \gamma\in\Gamma \colon B_\gamma(t)-ct>u}
\le
P_\eta \pk{\exists \, t\in[0,T]\colon  B(t)-ct>u},
\]
where $P_\eta$ is the total amount of the branches in the decision tree.
It appears that, as $u\to\infty$, the above bound provides the exact asymptotics, as
shown in the following theorem.
\begin{theo}\label{single_ruin}
\kk{For} $c\in\R$, as $u\to\infty$, 
	\begin{eqnarray*}
		\pk{\exists \, t\in[0,T], \, \exists \, \gamma\in\Gamma \colon B_\gamma(t)-ct>u}
		&\sim&
P_\eta \pk{\exists \, t\in[0,T]\colon  B(t)-ct>u}\\
&\sim&
P_\eta\sqrt{\frac{2}{\pi}}\frac{\sqrt{T}}{u+cT}
    \, \exp \left(-\frac{(u+cT)^2}{2T} \right).
	\end{eqnarray*}
\end{theo}

\subsection{The largest distance between the
branches}\label{s.span}
\kk{Next, we} investigate the probability that in time interval $[0,T]$
the largest distance between the
branches of the Brownian branching tree with drift
gets larger than $u>0$
\begin{eqnarray*}
\pk{\exists \, t\in[0,T], \exists \, \gamma_1,\gamma_2\in\Gamma \colon
(B_{\gamma_1}(t)-ct)-(B_{\gamma_2}(t)-ct)>u}.
\end{eqnarray*}

\kkk{We note that the drifts in the above formula cancel out.
Thus in the following result we consider the driftless case.}

\begin{theo}\label{diameter_1} As $u\to\infty$
	\begin{eqnarray*}
		\pk{\exists \, t\in[0,T], \exists \, \gamma_1,\gamma_2\in\Gamma \colon B_{\gamma_1}(t)-B_{\gamma_2}(t)>u}
		\sim P^2_\eta\frac{N_1-1}{N_1}\frac{2\sqrt{T-\tau_{1}}}{u\sqrt{\pi}}
		\exp\left(-\frac{u^2}{4(T-\tau_{1})}\right).
	\end{eqnarray*}
\end{theo}

\subsection{Simultaneous high-exceedance of all branches}\label{s.3}
\kk{In this section we analyze properties} of the simultaneous all-branch high exceedance probability
\begin{eqnarray}
\pk{\exists \, t\in[0,T]\ \forall \gamma\in\Gamma \colon B_\gamma(t)-ct>u}.\label{sim}
\end{eqnarray}
We begin with a non-asymptotic result. Obviously, for each $u> 0$
\begin{eqnarray}
\pk{\exists \, t\in[0,T]\ \forall \gamma\in\Gamma \colon B_\gamma(t)-ct>u}
\geq
\pk{\forall \gamma\in\Gamma \colon B_\gamma(T)-cT>u},\label{low_all}
\end{eqnarray}
where the asymptotics, as \( u \to \infty \), of the probability on the righthand side of the above inequality is given
in~\neprop{w(t)_asympt}.

In the following proposition we find an upper bound for (\ref{sim}) which differs from the
bound (\ref{low_all}) by some constant.

\begin{prop}\label{Korsh} \kk{For $c\in\R$ and} any $u>0$,
	\begin{align*}
		\pk{\exists \, t\in[0,T]\ \forall \gamma\in\Gamma \colon B_\gamma(t)-ct>u}
\leq C \,
\pk{\forall \gamma\in\Gamma \colon B_\gamma(T)-cT>u},
	\end{align*}
	where $C=\left(\pk{\xi>\abs{c}}\right)^{\kk{-P_\eta}}$ with $\xi$ a standard normal random
	variable.
\end{prop}

In order to present the exact asymptotics of
(\ref{sim}) as $u\to\infty$ we need to introduce the following constant
\begin{align}\label{Pit_constant_formula}
	\mathcal{H}_{\mathcal{N},\mathcal{\lambda}}
	\coloneqq
	\int_{\R^{\mathcal{N}}}
	\pk{\exists \, t\in(0,\infty), \, \forall \, i\in\{1,\ldots,\mathcal{N}\} \colon \lambda B_i^*(t)-\lambda^2 t>x_i}
	e^{\sum_{i=1}^{\mathcal{N}}x_i}\td \vk x,
\end{align}
where $B_i^*(t),\ t\in [0,\infty)$, $i=1,..., \mathcal{N}\in\N$
are mutually independent standard Brownian motions and
$\lambda>0$.
The finiteness of this constant is shown in~\nelem{Pit_constant}.

The following result constitutes the main finding of this section.
\begin{theo}\label{main} 
\kk{For $c\in\R$,}  as $u\to\infty$,
	\begin{equation*}
		\pk{\exists \, t\in[0,T] \ \forall \, \gamma\in\Gamma \colon B_\gamma(t)-ct>u}
		\sim \mathcal{H}_{P_\eta,1/\mu_0(T)} \, \pk{\forall \, \gamma\in\Gamma \colon B_{\kk{\gamma}}(T)-c T>u}.
	\end{equation*}
\end{theo}

\begin{remark}
	The results obtained in \neprop{Korsh} and \netheo{main} still hold for
	$T=T_u$ with $T_u\to T$ as $u\to\infty$.
\end{remark}

Interestingly, the above result can be extended to the version of Brownian decision tree
with random numbers of offsprings $N_i$. For a random vector
\( \vk{N} \), let \( \operatorname{essinf} ( \vk{N} ) \) denote the componentwise
essential infimum of \( \vk{N} \).
\begin{korr}\label{random_N}
	Assume that $N_i\in\N$ are independent random variables. Then,
	\begin{eqnarray*}
		\pk{\exists \, t\in[0,T] \ \forall \gamma\in\Gamma \colon ~B_\gamma(t)-ct>u}
		&\sim& \mathcal{H}_{P_\eta,1/\mu_0(T)}
		\prod_{i=1}^{\eta} \pk{N_i=\operatorname{essinf} (N_i)}\\
		&&\times \pk{\forall \,\gamma\in\Gamma \colon B_\gamma(T)-ct>u\ \mid \vk{N} = \operatorname{essinf} ( \vk{N} )},
	\end{eqnarray*}
	as \( u \to \infty \). The constant $\mathcal{H}_{P_\eta,1/\mu_0(T)}$ is calculated under the
	assumption that $\vk N=\operatorname{essinf} (\vk{N})$.
\end{korr}

The results obtained in~\netheo{main} and~\neprop{w(t)_asympt} allow us to
derive the asymptotics for the conditional \kk{first time of simultaneous exceedance of all branches}
\begin{equation*}
		\mathcal{T}(u):=\inf\{t\in[0,T] \colon \forall \, \gamma\in\Gamma \colon B_\gamma(t)-ct>u\}.
	\end{equation*}

\begin{korr}\label{ruintime}
%
	Then for any $x,y\in(0,+\infty)$ such that $x>y$, holds
	\begin{equation*}
		\lim_{u \to \infty}
		\pk{u^2(T-\mathcal{T}(u)) \geq x \mid \mathcal{T}(u) \leq T - y/u^2}
		= \exp\left(-\frac{(x-y)P_\eta}{2\mu_0^2(T)}\right).
	\end{equation*}
\end{korr}

\begin{example}[Binary tree]
	\kk{Suppose that the difference between branching points equals one
	(i.e. $\tau_{i}-\tau_{i-1}=1$ for all $i\in\N$) and each process always splits into
	two (i.e. $N_i=2$ for all $i\in\N$). For the sake of simplicity, let \( T \) be integer.
Then
\begin{align*}
	\mu_v(T)=2^{T-v}-1.
\end{align*}
}
Combining \netheo{main} and \neprop{w(t)_asympt} we obtain that as \( u \to \infty \)
\begin{eqnarray*}
	\pk{\exists \, t\in[0,T] \ \forall\gamma\in\Gamma  \colon B_\gamma(t)-ct>u}
	\sim
		\frac{\mathcal{H}_{2^{T-1},1/(2^{T}-1)}(2^T-1)^{2^{T-1}-1/2}}{(2\pi)^{2^{T-2}}
	\prod_{v=1}^{T-1}(2^{T-v}-1)^{2^{v-2}}}
u^{-2^{T-1}}	
\exp\left(-\frac{2^{T-2}}{2^{T}-1}(u+cT)^2\right).
\end{eqnarray*}
\end{example}

\subsection{Brownian decision forests}
\label{sec:forest}
In this section we shall consider a set of $M\in\N$ independent Brownian decision
trees with drift. That is, for \( i = 1, \ldots, M \) we take a triple
\( ( \vk{\tau}_i, \vk{N}_i, c ) \), where \( \vk{\tau}_i \) and \( \vk{N}_i \) are two
vectors and \( c_i \) is a constant, associate to each of them a Brownian
decision tree \( \vk{B}_{\Gamma_i} \) as described in Section~\ref{sec:definitions} (such that all of them
\kk{are} independent), and set
\begin{equation*}
	\vk{W}_i ( t )
	\coloneqq
	\vk{B}_{\Gamma_i} ( t ) + ( x_i - c_i \, t ) \, \vk{1}_{\eta_i},
\end{equation*}
where \( x_i \)'s are interpreted as the starting points of these trees. Each
tree \( \vk{W}_i ( t ) \) is thus uniquely defined by its tuple
$(\vk\tau_i,\vk N_i, c_i, x_i)$. Let us define a partial order relation $\succcurlyeq$ on the
set of trees as follows: we write
\( (\vk\tau_1,\vk N_1,c_1,x_1)\succcurlyeq(\vk\tau_2,\vk N_2,c_2,x_2) \) if one of the following
three conditions holds:
\begin{equation*}
	\begin{aligned}
		\text{(i)} \qquad
		&
			\frac{\mu_{0,1}(T)}{P_{\eta_1}}
			> \frac{\mu_{0,2}(T)}{P_{\eta_2}} \, ,
		\\[7pt]
		\text{(ii)} \qquad
		&
			\frac{\mu_{0,1}(T)}{P_{\eta_1}}
			= \frac{\mu_{0,2}(T)}{P_{\eta_2}}
			\quad \text{and} \quad
			c_1 T-x_1 < c_2 T-x_2 \, ,
		\\[9pt]
		\text{(iii)} \qquad
		&
			\frac{\mu_{0,1}(T)}{P_{\eta_1}}
			= \frac{\mu_{0,2}(T)}{P_{\eta_2}}
			\quad \text{and} \quad
			c_1 T-x_1 = c_2 T-x_2
			\quad \text{and} \quad
			P_{\eta_1} \leq P_{\eta_2}.
	\end{aligned}
\end{equation*}
In other words, it is the lexicographic order on the tuples
\( ( \mu_0 ( T ) / P_{\eta}, \, x - c T, \, -P_{\eta} ) \). One may notice that this
order is full. Next, let
\begin{equation*}
	(\vk\tau_1,\vk N_1,c_1,x_1)\approx(\vk\tau_2,\vk N_2,c_2,x_2)
	\iff
	\begin{cases}
		(\vk\tau_1,\vk N_1,c_1,x_1) \preccurlyeq (\vk\tau_2,\vk N_2,c_2,x_2),
		\\
		(\vk\tau_1,\vk N_1,c_1,x_1) \succcurlyeq (\vk\tau_2,\vk N_2,c_2,x_2)
	\end{cases}
\end{equation*}
and
\begin{equation*}
	(\vk\tau_1,\vk N_1,c_1,x_1) \succ (\vk\tau_2,\vk N_2,c_2,x_2)
	\iff
	\begin{cases}
		(\vk\tau_1,\vk N_1,c_1,x_1) \succcurlyeq (\vk\tau_2,\vk N_2,c_2,x_2),
		\\
		(\vk\tau_1,\vk N_1,c_1,x_1)\not\approx(\vk\tau_2,\vk N_2,c_2,x_2).
	\end{cases}
\end{equation*}
Combining \netheo{main} and \neprop{w(t)_asympt} we straightforwardly obtain the
following result.
\begin{lem}\label{tree_compare}
	The following equivalences hold:
	\begin{align*}
		&(\vk\tau_i,\vk N_i,c_i,x_i)\succ(\vk\tau_j,\vk N_j,c_j,x_j) \iff \lim_{u\to\infty}\frac{\pk{\exists t\in[0,T]:~\vk W_i(t)>u\vk 1_{\eta_i}}}{\pk{\exists t\in[0,T]:~\vk W_j(t)>u\vk 1_{\eta_j}}}=\infty,\\
		&(\vk\tau_i,\vk N_i,c_i,x_i)\prec(\vk\tau_j,\vk N_j,c_j,x_j) \iff \lim_{u\to\infty}\frac{\pk{\exists t\in[0,T]:~\vk W_i(t)>u\vk 1_{\eta_i}}}{\pk{\exists t\in[0,T]:~\vk W_j(t)>u\vk 1_{\eta_j}}}=0,\\
		&(\vk\tau_i,\vk N_i,c_i,x_i)\approx(\vk\tau_j,\vk N_j,c_j,x_j) \iff \lim_{u\to\infty}\frac{\pk{\exists t\in[0,T]:~\vk W_i(t)>u\vk 1_{\eta_i}}}{\pk{\exists t\in[0,T]:~\vk W_j(t)>u\vk 1_{\eta_j}}}\in(0,\infty).
	\end{align*}
\end{lem}
\kd{Lemma \ref{tree_compare} allows us to} find the asymptotics for the probability that all
branches of at least one of our $M$ trees exceed the threshold $u$.

\begin{theo}\label{forest_ruin}
	Let us define by $A$ the set of indices $i\in\{1,\ldots,M\}$ for which the
	corresponding set $(\vk \tau_i,\vk N_i, c_i,x_i)$ is the maximal among all given.
	Then, as $u\to\infty$,
	\begin{align*}
		\pk{\exists \, i\in\{1,\ldots,M\}, t\in[0,T] \colon \vk W_i(t)>u\vk 1_{\eta_i}}
		\sim \sum_{i\in A}\pk{\exists \, t\in[0,T] \colon \vk W_i(t)>u\vk 1_{\eta_i}},
	\end{align*}
	where the asymptotics of each term is given in \netheo{main}.
\end{theo}

\section{Proofs}
\label{sec:proofs}

\subsection{Proof of Proposition~\ref{classical}}
	Let $\tau \sim \operatorname{Exp} ( 1 )$ be the first branching moment. Define by
	$\mathfrak{d}$ the dimension of vector $\vk B(t)$ (i.e. the amount of branches at the
	point $t$), and $\vk 1(t)=(1,\ldots,1)\in\R^{\mathfrak{d}}$. Then,
	\begin{align*}
		&
			\pk{\exists \, t\in[0,T] \colon \vk B(t)-ct\vk 1(t)>u\vk 1(t), \  \tau\geq T}
		\\[7pt]
		& \hspace{50pt}
			\leq \pk{\exists \, t\in[0,T] \colon \vk B(t)-ct\vk 1(t)>u\vk 1(t)}
		\\[7pt]
		& \hspace{50pt}
			\leq \pk{\exists \, t\in[0,T] \colon \vk B(t)-ct\vk 1(t)>u\vk 1(t), \  \tau\leq T-2u^{-1}}
		\\[7pt]
		& \hspace{80pt}
			+\pk{\exists \, t\in[0,T] \colon \vk B(t)-ct\vk 1(t)>u\vk 1(t), \ \tau\in[T-2u^{-1},T]}
		\\[7pt]
		& \hspace{80pt}
			+\pk{\exists t \, \in [0,T] \colon \vk B(t)-ct\vk 1(t)>u\vk 1(t), \ \tau\geq T}
		\\[7pt]
		& \hspace{50pt}
			=: S_1+S_2+S_3
	\end{align*}
	Consider each term separately. For $S_3$, using that $ (\vk B\mid \tau\geq T )$ for
	$t\in[0,T]$ is a Brownian motion independent of $\tau$, we have
	\begin{align*}
		S_3
		&=\pk{\tau\geq T} \, \pk{\exists \, t\in[0,T] \colon \vk B (t)-ct\vk 1(t)>u\vk 1(t) \mid \tau\geq T}
		\\[7pt]
		&\sim
			e^{-T} \sqrt{\frac{2 T}{\pi}} \,
			u^{-1} \exp\left(-\frac{(u+cT)^2}{2T}\right)
	\end{align*}
	as $u\to\infty$. Considering $S_2$, again using that $B_1$ is Brownian motion independent of $\tau$, as $u\to\infty$
	\begin{align*}
		S_2
		&\leq \pk{\tau\in[T-2u^{-1},T]} \, \pk{\exists \, t\in[0,T] \colon B_1(t)-ct>u}
		\\[7pt]
		&\sim
			(e^{2u^{-1}}-1) \,
			e^{-T} \, \sqrt{\frac{2 T}{\pi}} \, u^{-1}
			\exp\left(-\frac{(u+cT)^2}{2T}\right),
	\end{align*}
	implies that
	\begin{equation*}
		\lim_{u \to \infty}
		\frac{S_2}{S_3}\to 0.
	\end{equation*}
	Then, for $S_1$
	\begin{eqnarray*}
		S_1
		&\leq&
        \pk{\exists \, t\in[0,T-u^{-1}] \colon B_1(t)-ct>u, \ \tau\leq T-2u^{-1}}
		\\
		&&+\pk{\exists \, t\in[T-u^{-1},T] \colon \frac{B_1(t)+B_2(t)}{2}-ct>u, \ \tau\leq T-2u^{-1}}
		\eqqcolon Z_1+Z_2.
	\end{eqnarray*}
	For $Z_1$, using that
	\begin{equation*}
		Z_1
		\sim
		( 1 - e^{-T} ) \,
		\sqrt{\frac{2 T}{\pi}} \,
		u^{-1}
		\exp\left(-\frac{(u+c(T-u^{-1}))^2}{2(T-u^{-1})}\right),
	\end{equation*}
	we obtain
	\begin{eqnarray*}
		\lim_{u \to \infty}
		\frac{Z_1}{S_3}
		&=& \lim_{u \to \infty}
		\frac{1-e^{-T}}{e^{-T}}
		\exp\left(\frac{(u+cT)^2}{2T}-\frac{(u+c(T-u^{-1}))^2}{2(T-u^{-1})}\right)
		\\
		&=& \lim_{u \to \infty}
		\frac{1-e^{-T}}{e^{-T}}
		\exp\left(\frac{-u^{-1}(u+cT)^2-4cT^2+2Tc^{2}u^{-2}}{2T(T-u^{-1})}\right)
		= 0.
	\end{eqnarray*}
	Finally,
	\begin{eqnarray*}
		Z_2
		&=&\E{
			\mathbb{I}_{\{\tau<T-2u^{-1}\}} \,
			\pk{\exists \, t\in[T-u^{-1},T] \colon \frac{B_1(t)+B_2(t)}{2}-ct>u \  \middle| \  \tau}
		}
		\\
		&\leq&\sup_{\mathfrak{t}\in[0,T-2u^{-1}]}
		\pk{\exists \, t\in[T-u^{-1},T] \colon \frac{B_1(t)+B_2(t)}{2}>u-\abs{c}T \ \middle| \  \tau=\mathfrak{t}}.
	\end{eqnarray*}
	Since $( B_1(t)+B_2(t)\mid\tau=\mathfrak{t} )$ is a Gaussian process and for any $\mathfrak{t}\in[0,T-u^{-1}]$
	\begin{equation*}
		\Var\left(\frac{B_1(t)+B_2(t)}{2} \ \middle| \ \tau=\mathfrak{t}\right)
		=\begin{cases}
			 \frac{t+\mathfrak{t}}{2}, & t>\mathfrak{t},\\
			 t, & t\leq \mathfrak{t},
		 \end{cases}
	 \end{equation*}
	 we can apply Piterbarg inequality (see, e.g.,~\cite[Theorem 8.1]{Pit96})
	 obtaining that for $u>\abs{c}T$
	 \begin{eqnarray*}
		 \pk{\exists \, t\in[T-u^{-1},T] \colon \frac{B_1(t)+B_2(t)}{2}>u-\abs{c}T \ \middle| \ \tau=\mathfrak{t}}
		 \leq
		 C(u-\abs{c}T)^{\alpha}\exp\left(-\frac{(u-\abs{c}T)^2}{2\sigma_{\mathfrak{t}}}\right),
	 \end{eqnarray*}
	 where
	 \begin{equation*}
		 \sigma_{\mathfrak{t}}=\max_{t\in[0,T]}\Var\left(\frac{B_1(t)+B_2(t)}{2}\right)=\frac{T+\mathfrak{t}}{2}
	 \end{equation*}
	 and $C>0, \, \alpha\in\R$ are some constants. Using that $C, \, \alpha$ does not depend ot
	 $\mathfrak{t}$, we get
	 \begin{equation*}
		 Z_2\leq Cu^{\alpha} \sup_{\mathfrak{t}\in[0,T-2u^{-1}]}
		 \exp\left(-\frac{(u-\abs{c}T)^2}{2\sigma_{\mathfrak{t}}}\right)
		 = Cu^{\alpha}\exp\left(-\frac{(u-\abs{c}T)^2}{2(T-u^{-1})}\right).
	 \end{equation*}
	 It remains to note that \( Z_2/S_3\to 0 \) as \( u \to \infty \) by the same reason as
	 $Z_1/S_3$.
\QED

\subsection{Proof ot Theorem~\ref{single_ruin}}
\kk{We begin with the observation that
\begin{eqnarray*}
	\pk{\exists \, t\in[0,T], \gamma\in\Gamma \colon B_\gamma(t)-ct>u}
&\geq&
	\sum_{\gamma\in\Gamma} \pk{\exists \, t\in[0,T] \colon B_\gamma(t)-ct>u}\\
	&&-\sum_{\substack{\gamma_1,\gamma_2\in\Gamma\\ \gamma_1\not=\gamma_2}}
	\pk{
		\exists \, t\in[0,T] \colon
		\begin{aligned}
			B_{\gamma_1}(t)-ct>u, \\[3pt]
			B_{\gamma_2}(t)-ct>u
		\end{aligned}
	}
\end{eqnarray*}
and
\begin{equation*}
	\pk{\exists \, t\in[0,T], \gamma\in\Gamma \colon B_\gamma(t)-ct>u}
	\leq
	\sum_{\gamma\in\Gamma}\pk{\exists \, t\in[0,T] \colon B_\gamma(t)-ct>u},
\end{equation*}
}
where branch $B_\gamma$ is a Brownian motion. By~\cite[Formula 1.1.4 and Appendix 2,
Section 8]{borodin2015handbook}, we have
\begin{equation}\label{eq:5}
	\pk{\exists \, t\in[0,T] \colon B_\gamma(t)-ct>u}
	\sim \frac{2\sqrt{T}}{\sqrt{2\pi}(u+cT)} \,
	\exp \left( -\frac{(u+cT)^2}{2T} \right)
\end{equation}
as $u\to\infty$. For the double sum on the left hand side, we have the following upper
bound
\begin{equation*}
	\pk{
		\exists \, t\in[0,T] \colon
		\begin{aligned}
			B_{\gamma_1}(t)-ct>u, \\[3pt]
			B_{\gamma_2}(t)-ct>u
		\end{aligned}
	}
	\leq C \, \pk{
		\begin{aligned}
			B_{\gamma_1}(T)-cT>u, \\[3pt]
			B_{\gamma_2}(T)-cT>u
		\end{aligned}
	}.
\end{equation*}
Here $C$ is some positive constant independent of $u$; see ~\cite[Theorem 3.1]{MR4569300}.
Since $(B_{\gamma_1}(T),B_{\gamma_2}(T))$ is a Gaussian vector with
covariance matrix
\begin{align*}
	\Sigma=\begin{pmatrix} T & \tau_{\kappa(\gamma_1,\gamma_2)}\\
		\tau_{\kappa(\gamma_1,\gamma_2)} & T
	\end{pmatrix}
\end{align*}
we have for some positive constant $C_2>0$ as \( u \to \infty \)
\begin{align*}
	\pk{B_{\gamma_1}(T)+cT>u, \, B_{\gamma_2}(T)-cT>u}
	\sim C_2 \, u^{-2} \,
	\exp \left( -\frac{(u+cT)^2}{2 \left(T+\tau_{\kappa(\gamma_1,\gamma_2)}\right)} \right).
\end{align*}
Using that $\tau_{\kappa(\gamma_1,\gamma_2)}<T$ for $\gamma_1\not=\gamma_2$, we obtain by~\eqref{eq:5}
\begin{equation*}
	\pk{
		\exists \, t\in[0,T] \colon
		\begin{aligned}
			B_{\gamma_1}(t)-ct>u, \\
			B_{\gamma_2}(t)-ct>u
		\end{aligned}
	}
	=o\left(\pk{\exists \, t\in[0,T] \colon B_\gamma(t)-ct>u}\right)
\end{equation*}
for any $\gamma,\gamma_1,\gamma_2\in\Gamma$, $\gamma_1\not=\gamma_2$, establishing the claim.
\QED

\subsection{Proof of Theorem~\ref{diameter_1}}
Note that
\begin{align*}
	& \sum_{\gamma_1,\gamma_2\in\Gamma} \pk{\exists \, t \in[0,T] \colon \,B_{\gamma_1}(t)-B_{\gamma_2}(t)>u}
  \\ & \hspace{30pt}
		\geq \pk{\exists \, t\in[0,T], \exists \, \gamma_1,\gamma_2\in\Gamma \colon \,B_{\gamma_1}(t)-B_{\gamma_2}(t)>u}
  \\[7pt] & \hspace{30pt}
		\geq \sum_{\gamma_1,\gamma_2\in\Gamma}\pk{\exists \, t\in[0,T] \colon \,B_{\gamma_1}(t)-B_{\gamma_2}(t)>u}
	\\ & \hspace{50pt}
	 -\sum_{\substack{\gamma_1,\gamma_2,\gamma_3,\gamma_4\in\Gamma\\ \{\gamma_1,\gamma_2\}\not=\{\gamma_3,\gamma_4\}}}
	\pk{
	\exists \, t,s\in[0,T] \colon
	B_{\gamma_1}(t)-B_{\gamma_2}(t)>u, B_{\gamma_3}(s)-B_{\gamma_4}(s)>u
	}.
\end{align*}
For any $\gamma_1\not=\gamma_2$, the process $B_{\gamma_1}-B_{\gamma_2}$ admits the following representation
\begin{equation*}
	B_{\gamma_1}(t)-B_{\gamma_2}(t)
	=\begin{cases}
		 0,\qquad &t\leq\tau_{\kappa(\gamma_1,\gamma_2)}, \\
		 B^*(2(t-\tau_{\kappa(\gamma_1,\gamma_2)}),\qquad &t>\tau_{\kappa(\gamma_1,\gamma_2)},
	\end{cases}
\end{equation*}
where $B^*(t)$ is Brownian motion. Hence, as $u\to\infty$,
\begin{eqnarray*}
	\pk{\exists \, t\in[0,T] \colon B_{\gamma_1}(t)-B_{\gamma_2}>u}
		&=&
\pk{\exists \, t\in[\tau_{\kappa(\gamma_1,\gamma_2)},T] \colon B^*(2(t-\tau_{\kappa(\gamma_1-\gamma_2)}))>u}
	\\
	&=&\pk{\exists \, t\in[0,2(T-\tau_{\kappa(\gamma_1,\gamma_2)})] \colon B^*(t)>u}\\
	&\sim& \frac{2\sqrt{2(T-\tau_{\kappa(\gamma_1,\gamma_2)})}}{u\sqrt{2\pi}}
	\exp\left(-\frac{u^2}{4(T-\tau_{\kappa(\gamma_1,\gamma_2)})}\right).
\end{eqnarray*}
For  $\gamma_1,\gamma_2,\gamma_3,\gamma_4\in\Gamma$, if $\kappa(\gamma_1,\gamma_2)<\kappa(\gamma_3,\gamma_4)$, then as $u\to\infty$
\begin{equation*}
	\frac{
		\pk{\exists \, t\in[0,T] \colon B_{\gamma_3}(t)-B_{\gamma_4}>u}
	}{
		\pk{\exists \, t\in[0,T] \colon B_{\gamma_1}(t)-B_{\gamma_2}>u}
	}\to 0.
\end{equation*}
The latter implies that
\begin{eqnarray*}
	\sum_{\gamma_1,\gamma_2\in\Gamma}
	\pk{\exists \, t\in[0,T] \colon B_{\gamma_1}(t)-B_{\gamma_2}(t)>u}
	&\sim& \sum_{\substack{\gamma_1,\gamma_2\in\Gamma\\\kappa(\gamma_1,\gamma_2)=1}}
	\frac{2\sqrt{2(T-\tau_{1})}}{u\sqrt{2\pi}}\exp\left(-\frac{u^2}{4(T-\tau_{1})}\right)
	\\
		&=& N_1(N_1-1)\left(\frac{P_\eta}{N_1}\right)^2
		\frac{2\sqrt{2(T-\tau_{1})}}{u\sqrt{2\pi}}
		\exp\left(-\frac{u^2}{4(T-\tau_{1})}\right).
\end{eqnarray*}
Next, consider the double events' probabilities. For  $\gamma_1,\gamma_2,\gamma_3,\gamma_4\in\Gamma$,
\begin{eqnarray*}
\lefteqn{	\pk{\exists \, t,s\in[0,T] \colon B_{\gamma_1}(t)-B_{\gamma_2}(t)>u, B_{\gamma_3}(s)-B_{\gamma_4}(s)>u}}\\
	&&\leq \pk{\exists \, t,s\in[0,T] \colon \frac{B_{\gamma_1}(t)-B_{\gamma_2}(t)+B_{\gamma_3}(s)-B_{\gamma_4}(s)}{2}>u}.
\end{eqnarray*}
Using that $(B_{\gamma_1}(t)-B_{\gamma_2}(t)+B_{\gamma_3}(s)-B_{\gamma_4}(s))/2$ is a Gaussian random field and
\begin{equation*}
	\Var\left(\frac{B_{\gamma_1}(t)-B_{\gamma_2}(t)+B_{\gamma_3}(s)-B_{\gamma_4}(s)}{2}\right)
  \leq\frac{3T+\tau_{\eta} - 4\tau_1}{2},
\end{equation*}
by Piterbarg inequality (see, e.g.,~\cite[Theorem 8.1]{Pit96}) there exist some
constants $C>0, \ \alpha\in\R$
\begin{eqnarray*}
	\pk{\exists \, t,s\in[0,T] \colon
	B_{\gamma_1}(t)-B_{\gamma_2}(t)+B_{\gamma_3}(s)-B_{\gamma_4}(s) > 2 u}
	\leq Cu^{\alpha}\exp\left(-\frac{u^2}{3T+\tau_\eta-4\tau_1}\right).
\end{eqnarray*}
The above inequality implies that
\begin{equation*}
	\frac{
		\displaystyle
		\sum_{\substack{\gamma_1,\gamma_2,\gamma_3,\gamma_4\in\Gamma\\ \{\gamma_1,\gamma_2\}\not=\{\gamma_3,\gamma_4\}}}
		\pk{\exists \, t,s\in[0,T] \colon
			\begin{aligned}
				& B_{\gamma_1}(t)-B_{\gamma_2}(t)>u,
				\\[3pt]
				& B_{\gamma_3}(s)-B_{\gamma_4}(s)>u
			\end{aligned}
			}
		}{
			\displaystyle
			\sum_{\gamma_1,\gamma_2\in\Gamma}\pk{\exists t\in[0,T]:\,B_{\gamma_1}(t)-B_{\gamma_2}(t)>u}
		}\to 0
	\end{equation*}
	as $u\to\infty$, establishing the proof.
\QED

\subsection{Proofs of results on simultaneous high-exceedance probability of all branches}
	Until the end of this secion, let us define
	\begin{equation}
		\vk{W} ( t ) = \vk{B}_{\Gamma} ( t ) - \vk{c}  t ,
		\qquad
		\vk{c}  = c  \, \vk{1}_{\eta},
		\label{W_def}
	\end{equation}
	where $c\in\R$ is fixed. Then the the all-branch high exceedance probability can be stated as
	\begin{align*}
		\pk{\exists \, t\in[0,T], \, \forall \gamma\in\Gamma \colon \,B_\gamma(t)-ct>u}=\pk{\exists \, t\in[0,T] \colon \vk W(t)>u\vk 1_{\eta}}.
	\end{align*}

\begin{proof}[Proof of Proposition~\ref{Korsh}]
	As in the proof of~\cite[Theorem 1.1]{korshunov2020tail} define a stopping
	time by
\begin{align*}
	\mathcal{T} = \inf\{t \in [0,T] \colon \vk W(t)>u\vk 1_{\eta}\}
	=\inf\{t\in[0,T] \colon \forall\gamma\in\Gamma~B_\gamma(t)-ct>u\},
\end{align*}
using the convention $\inf\varnothing=T+1$ and write $F_{\mathcal{T}}$ for its
distribution function. Since the paths of $\vk W$ are continuous, $\mathcal{T}$ is
well-defined. By Proposition~\ref{ind_increments},
\begin{align*}
	\pk{\vk W(T)>u\vk 1_{\eta}}
	&=\int_0^{T}dF_{\mathcal{T}}(t) \, \pk{\vk W(T)>u\vk 1_{\eta} \mid \mathcal{T}=t}\\[7pt]
	&=\int_0^{T}dF_{\mathcal{T}}(t) \, \pk{\forall\gamma\in\Gamma \colon B_\gamma(T)-cT>u \mid \mathcal{T}=t}\\[7pt]
	&\geq\int_0^{T}dF_{\mathcal{T}}(t) \, \pk{\forall\gamma\in\Gamma \colon B_\gamma(T)-B_\gamma(t)>c(T-t)}.
\end{align*}
Consider the last probability. Define for $t\in[0,T_1]$ and $\gamma\in\Gamma$ a Gaussian
random variable $\overline{B}_{\gamma}(t)$ by
\begin{equation*}
	\overline{B}_{\gamma}(t)=\left(B_{\gamma}(T)-B_{\gamma}(t)\right).
\end{equation*}
Its variance is
\begin{equation*}
	\Var\left(\overline{B}_{\gamma}(t)\right)
  = \Var\left(B_{\gamma}(T)-B_{\gamma}(t)\right)=T_1-t,
\end{equation*}
and the covariance between them
\begin{equation*}
	\Cov \left(\overline{B}_{\gamma_1}(t), \overline{B}_{\gamma_2}(t)\right)
	= \min \{ T, \tau_{\kappa(\gamma_1,\gamma_2)} \}
	-\min \{ t, \tau_{\kappa(\gamma_1,\gamma_2)} \}
	\geq 0.
\end{equation*}
Hence, using Slepian inequality (see e.g., Theorem 2.2.1 in~\cite{AdlerTaylor})
we arrive at
\begin{eqnarray*}
	\pk{\forall \, \gamma\in\Gamma \colon \overline{B}_{\gamma}(t)>c(T_1-t)}
	\geq
	\prod_{i=1}^{P_\eta}\pk{\overline{B}_i(t)>\abs{c}(T_1-t)}
	=\left(\pk{\xi>\abs{c}\sqrt{T_1-t}}\right)^{P_\eta}
	\geq \left(\pk{\xi>\abs{c}}\right)^{P_\eta},
\end{eqnarray*}
where $\xi$ is a standard normal random variable. Hence, we can continue our
initial inequality as follows
\begin{align*}
	\pk{\vk W(T_1)>u\vk 1_{\eta}}
	&\geq\int_0^{T_1}dF_{\mathcal{T}}(t) \, \pk{\forall\gamma\in\Gamma \colon B_\gamma(T_1)-B_\gamma(t)>c(T_1-t)}
	\\[7pt]
	&\geq\left(\pk{\xi>\abs{c}}\right)^{P_\eta}\int_0^{T_1}dF_{\mathcal{T}}(t) \\[7pt]
	&=\left(\pk{\xi>\abs{c}}\right)^{P_\eta}\pk{\mathcal{T}\leq T_1} \\[7pt]
	&=\left(\pk{\xi>\abs{c}}\right)^{P_\eta}\pk{\exists \, t\in[0,T_1]:\vk W(t)>u\vk 1_{\eta}}
\end{align*}
establishing the claim.
\end{proof}
\kkk{Before proving Theorem \ref{main} we need to introduce some useful notation and
present some lemmas that are used in the proof. }

In the remainder we shall need some results concerning the so-called quadratic
programming problem: for a positive-definite $d\times d$ matrix $\Sigma$ and real-valued
vector $\vk a\in\R^d$
\begin{equation}
	\label{eq:QP}
	\Pi_\Sigma(\vk{a}) \colon \quad
	\text{minimise} \quad
	\vk{x}^\top \Sigma^{-1} \vk{x}
	\quad \text{under the linear constraint} \quad
	\vk{x} \ge \vk{a}.
\end{equation}
As shown in~\cite{Hag79}, \eqref{eq:QP} has a unique solution $\tilde{\vk a}$,
see Appendix, \nelem{lem:quadratic_programming}. We would also use further the
sets $I,J$ defined in \nelem{lem:quadratic_programming}. In particular,
$ \tilde{\vk a}_I= \vk a_I$ with some index set $I$. The components of
$\tilde{\vk a}$ with indices in $J=\{1 \ldot d\} \setminus I$ can be expressed in terms
of $\vk a_I$. The next result shows the solution of $\Pi_{\Sigma(T)}(\vk 1_{\eta})$ for
our model.
\begin{lem}\label{quadratic_solution_prop}
	The unique solution of $\Pi_{\Sigma(T)}(\vk 1_{\eta})$ satisfies
	\begin{align*}
		\tilde{\vk a}=\vk 1_{\eta},\qquad I=\{1,\ldots,P_{\eta}\},\qquad J=\varnothing.
	\end{align*}
\end{lem}
The proof of \nelem{quadratic_solution_prop} is given in Appendix.

Let $\delta_u(L)=Lu^{-2}$ for $L>0$ and denote
into two
\begin{align*}
	M(u,L)& \coloneqq \pk{\exists \, t\in[T-\delta_u(L),T] \colon \vk W(t)>u\vk 1_{\eta}},\\
	m(u,L)& \coloneqq \pk{\exists \, t\in[0,T-\delta_u(L)] \colon \vk W(t)>u\vk 1_{\eta}}.
\end{align*}
Clearly,
\begin{align*}
	M(u,L) \leq \pk{\exists \, t\in[0,T] \colon \vk W(t)>u\vk 1_{\eta}}\leq M(u,L)+m(u,L).
\end{align*}
We will prove that $m(u,L)$ is negligible with respect to $M(u,L)$.

\begin{lem}\label{Pickands}
	For  fixed $L>0$, as $u\to\infty$
	\begin{align*}
		M(u,L)\sim H(L) \, \pk{\vk W(T)>u\vk 1_{\eta}},
	\end{align*}
	where $\vk\lambda=\Sigma^{-1}(T)\vk 1_\eta$, and the constant $H(L)$ is given by
\begin{equation}
	\label{H(L)}
	H(L) = e^{-\frac{L}{2}\vk \lambda^\top\vk\lambda}
	\int_{\R^{P_\eta}}
	\pk{\exists \, t\in[0,L] \colon \vk\lambda\vk B^*(t)>\vk x}
	\, e^{\sum_{i=1}^{P_\eta} x_i}
	\td\vk x,
\end{equation}
where $\vk B^*(t)\in\R^{P_\eta}$ is a Brownian motion with independent components.
\begin{remark}
	In view of \nelem{quadratic_solution_prop} and
	\nelem{lem:quadratic_programming} (see Appendix) it follows that $\vk\lambda>\vk 0$.
\end{remark}
\end{lem}

\begin{proof}[Proof of Lemma~\ref{Pickands}]
	Let $\varphi_{t}$ denote the pdf of $\vk W(t)$. We can assume that $u$ is large
	enough to ensure that $i(T-\delta_u(L))=i(T)$. Then,
\begin{equation}\label{eq:6}
	M ( u, L )
	= u^{-P_{\eta}} \int_{\mathbb{R}^{P_{\eta}}}
  J ( u, T, L, \vk{x} ) \,
	\varphi_{T - \delta_u ( L )} \left( u \, \vk{1}_{\eta} - \frac{\vk{x}}{u} \right)
	\mathop{d \vk{x}},
\end{equation}
where
\begin{equation*}
  J ( u, T, L, \vk{x} )
	\coloneqq
  \mathbb{P} \left\{
		\exists \, t \in [T - \delta_u ( L ), T ] \colon
		\vk{W} ( t ) > u \, \vk{1}_{\eta}
		\mid
		\vk{W} ( T - \delta_u ( L ) ) = u \, \vk{1}_{\eta} - \frac{\vk{x}}{u}
	\right\}.
\end{equation*}
Next, using \neprop{ind_increments} we can rewrite the probability
\( J ( u, T, L, \vk{x} ) \) as follows
\begin{align*}
  J ( u, T, L, \vk{x} )
	& =
		\mathbb{P} \left\{
    \exists \, t \in [ T - \delta_u ( L ), T ] \colon
		\vk{W} ( t ) - \vk W ( T - \delta_u ( L ) )
		> \frac{\vk{x}}{u}
		\right\}
	\\[7pt]
	& =
		\mathbb{P} \left\{
    \exists \, t \in [ T - \delta_u ( L ), T ] \colon
		\vk{B}^{ * } ( t ) - \vk{B}^{ * } ( T - \delta_u ( L ) )
		> \frac{\vk{x}}{u} + \vk{c} ( t - T + \delta_u ( L ) )
		\right\}
	\\[7pt]
	& =
		\mathbb{P} \left\{
    \exists \, t \in [ 0, \delta_u ( L ) ] \colon
		\vk{B}^{ * } ( t ) > \frac{\vk{x}}{u} + \vk{c} \, t
		\right\}
	\\[7pt]
	& =
	  \mathbb{P} \left\{
    \exists \, t \in [ 0, L ] \colon
		\vk{B}^{ * } ( t ) > \vk{x} + \frac{\vk{c} \, t}{u}
		\right\},
\end{align*}
where $\vk B^*(t)\in\R^{P_\eta}$ is a Brownian motion with independent components. We have
\begin{align*}
	\pk{\exists \, t \in [0,L] \colon \vk B^*(t) > \vk x + \vk{c} \, tu^{-1}}
  \xrightarrow[u \to \infty]{}
	\pk{\exists \, t \in [0,L] \colon \vk B^*(t) > \vk x}
\end{align*}
for almost all $\vk x\in\R^{P_\eta}$.

Next, let us consider the factor
\( \varphi_{T - \delta_u ( L )} \). We have
\begin{equation*}
	\varphi_{T - \delta_u ( L )} \left( u \, \vk{1}_{\eta} - \frac{\vk{x}}{u} \right)
	=
	( 2 \pi )^{-P_{\eta}/2} \left| \Sigma ( T - \delta_u ( L ) ) \right|^{-1/2}
	\exp ( -G ( u, T, L, \vk{x} ) ),
\end{equation*}
where
\begin{eqnarray*}
	G ( u, T, L, \vk{x} )
	\coloneqq
	-\frac{1}{2}
	\left(
		u \, \vk{1}_{\eta} - \frac{\vk{x}}{u} + \vk{c} ( T - \delta_u ( L ) )
	\right)^\top
	\times
	\Sigma^{-1} ( T - \delta_u ( L ) )
	\left(
		u \, \vk{1}_{\eta} - \frac{\vk{x}}{u} + \vk{c} ( T - \delta_u ( L ) )
	\right).
\end{eqnarray*}
Using
\begin{equation*}
	\Sigma^{-1} ( T - \delta_u ( L ) )
	= \Sigma^{-1} ( T )
	-\delta_u ( L ) \, \Sigma^{-2} ( T )
	+o \left( \delta_u^2 ( L ) \right),
\end{equation*}
we find that the prefactor converges to
\( ( 2 \pi )^{-P_{\eta} / 2} \left| \Sigma ( T ) \right|^{-1/2} \), and we can focus
on the asymptotics of the exponent. We have
\begin{eqnarray*}
	B ( u, T, L, \vk{x} )
	=
	\frac{1}{2}
	\left(
		u \, \vk{1}_{\eta} - \vk{c} T
	\right)^\top
	\Sigma^{-1} ( T )
	\left(
		u \, \vk{1}_{\eta} + \vk{c} \, T
	\right)
	+\vk{x}^\top \Sigma^{-1} ( T ) \, \vk{1}_{\eta}
	-\frac{L}{2} \, \vk{1}_{\eta}^\top \Sigma^{-2} ( T ) \vk{1}_{\eta}
	+O ( L / u ),
\end{eqnarray*}
\kk{where} 
we used \neprop{Sigma_rec}. Hence, as
$u\to\infty$,
\begin{align*}
	\varphi_{T-\delta_u(L)}(u\vk 1_\eta -\vk x/u)
	\sim \varphi_T(u\vk 1_\eta) \, e^{\vk x^\top\vk\lambda} \, e^{-\frac{L}{2}\vk \lambda^\top\vk\lambda}.
\end{align*}

Finally, we can bound the pdf under the integral~\eqref{eq:6}, for all large
enough $u$, as follows:
\begin{equation}
	\label{dominated_bound}
	\frac{\varphi_{T - \delta_u ( L )} ( u \vk{1}_{\eta} - \vk{x} / u )}{\varphi_T ( u \vk{1}_{\eta} )}
	\leq
	A \, \exp \left( \vk{x}^\top \vk{\lambda}_{\vk{x}} ( \varepsilon ) \right)
	\eqqcolon
	\bar{\varphi} ( \vk{x} ),
\end{equation}
where
\( \vk{\lambda}_{\vk{x}} ( \varepsilon ) \coloneqq \vk{\lambda} + \sign ( \vk{x} ) \, \varepsilon > \vk{0} \)
for all small enough \( \varepsilon \) and
\begin{equation*}
	A \coloneqq
	\max_{t \in [ ( \tau_{\eta} + T ) / 2, T ]} \sqrt{\frac{| \Sigma ( T ) |}{| \Sigma ( t ) |}}
	\times \max_{\substack{
			t \in [ ( \tau_{\eta} + T ) / 2, T ] \\[3pt]
			\vk{y}, \, \vk{z} \in [ -\vk{1}, \vk{1} ] \subset \R^{P_{\eta}}
		}}
	\exp \left( \vk{y}^\top \Sigma^{-1} ( t ) ( \vk{1}_{\eta} + \vk{z} ) \right)
	< \infty
\end{equation*}
In proving~\eqref{dominated_bound} we used the fact that both
$\vk{c} L u^{-1}$ and $\vk cTu^{-1}$ belong to $\in[-\vk 1,\vk 1]$ for large
enough $u$. The constant $A$ is clearly finite since $\Sigma(t)$ is continuous.

To find the asymptotics of $M(u,L)$, we apply dominated convergence theorem. The probability under the integral $J(u, T, L, \vk{x})$ can bounded
as follows. For  $\vk x\in\R^{P_\eta}$, define
\begin{equation*}
	F_+ ( \vk{x} ) \coloneqq \left\{ i \colon x_i > 0 \right\},
\end{equation*}
and two associated sets
\begin{equation*}
	\mathcal{S}_F \coloneqq \left\{ \vk{x} \in \R^{P_{\eta}} \colon F_+ ( \vk{x} ) = F \right\},
	\qquad
	\mathcal{C} \coloneqq \left\{ \vk{x} \in \R^d \colon \sum_{i \in F_+ ( \vk{x} )} ( x_i + \varepsilon ) < 0  \right\}.
\end{equation*}
By Piterbarg inequality (see, e.g.,~\cite[Theorem 8.1]{Pit96}) for each
$\vk x\in\mathcal{C}^c$ we have
\begin{align*}
J(u,T,L,\vk x)&\leq
	\pk{\exists \, t\in[0,L] \ \forall i\in F \colon B^*_i(t)>x_i+\varepsilon}\\
	&\leq
	\pk{\exists \, t\in[0,L] \colon \sum_{i\in F_+(\vk x)} B^*_{i}(t)
		> \sum_{i\in F_+(\vk x)}(x_i+\varepsilon)}
	\\[7pt]
	&\leq
	C \left(\sum_{i\in F_+(\vk x)}(x_i+\varepsilon)\right)^{\gamma}
	\exp\left( -\delta\sum\limits_{i\in F_+(\vk x)} (x_i+\varepsilon)^2 \right)
	\eqqcolon \bar R(\vk x)
\end{align*}
for some positive constants $C, \ \gamma, \ \delta$ and \( \varepsilon \) independent of
\( \vk{x} \). Thus, it is enough to show that the integral
$\int_{\R^{P_\eta}}\bar R(\vk x) \, \bar\varphi(\vk x)dx$ is finite. Setting
\begin{equation*}
	\mathcal{C}_F \coloneqq \left\{ \vk x\in \R^{P_{\eta}} \colon \vk x>\vk 0, \  \sum_{i\in F}(x_i+c)<0 \right\},
\end{equation*}
we obtain
\begin{eqnarray*}
	\int_{\R^{P_{\eta}}} \bar{R} ( \vk{x} ) \, \bar{\varphi} ( \vk{x} ) \mathop{d \vk{x}}
	&=&
	C A \sum_{F \subset \{ 1, \ldots, P_{\eta} \}}
	\Bigg[
	\int_{\mathcal{S}_F \setminus C} \left( \sum_{i \in F} ( x_i + \varepsilon ) \right)^{\gamma}
	\exp \left( \vk{x}^\top \vk{\lambda}_{\vk{x}} ( \varepsilon ) - \delta \sum_{i \in F} ( x_i + \varepsilon )^2 \right)
	\mathop{d \vk{x}}
	\\
	&&+\int_{\mathcal{S}_F \cap \mathcal{C}} \exp \left( \vk{x}^\top \vk{\lambda}_{\vk{x}} ( \varepsilon ) \right)
	\mathop{d \vk{x}}
	\Bigg]
	\eqqcolon
	C A \sum_{F \subset \{ 1, \ldots, P_{\eta} \}} \Big[ A_1 + A_2 \Big].
\end{eqnarray*}
The integral \( A_1 \) may be estimated as follows:
\begin{align*}
	& A_1 \leq
	\exp \left(
	-\frac{1}{4 \delta} \, \left\| \vk{\lambda}_F + \varepsilon \, \vk{1}_F \right\|_2^2
	-\varepsilon \, \vk{1}_F^\top \left( \vk{\lambda}_F + \varepsilon \, \vk{1}_F \right)
	\right)
	\\[7pt]
	& \hspace{30pt} \times
	\int_{\{ \vk{x}_F > \vk{0} \} \setminus \mathcal{C}_F}
	\exp \left( -\delta \left\|
	\vk{x}_F + \varepsilon \, \vk{1}_{F} - \frac{1}{2 \delta} \, \vk{\lambda}_{\vk{x}, F} ( \varepsilon )
	\right\|_2^2 \right)
	\left\| \vk{x}_F + \varepsilon \, \vk{1}_F \right\|_1^{\gamma}
	\mathop{d \vk{x}_F}
	\\[7pt]
	& \hspace{30pt} \times
	\int_{\vk{x}_{F^c} \leq \vk{0}_{F^c}}
	\exp \left( \vk{x}^\top_{F^c} \vk{\lambda}_{\vk{x}, F^c} ( \varepsilon ) \right)
	\mathop{d \vk{x}_{F^c}},
\end{align*}
where \( \left\| \, \vk{\cdot} \, \right\|_p \) denotes the \( \ell_p \) norm. Next, we
bound \( A_2 \) as follows:
\begin{equation*}
	A_2 \leq
	\int_{\vk{x}_F^c \leq \vk{0}_F^c}
	\exp \left( \vk{x}^\top_{F^c} \, \vk{\lambda}_{\vk{x}, F^c} ( \varepsilon ) \right)
	\mathop{d \vk{x}}
	\times \int_{\mathcal{C}_F} \mathop{d \vk{x}_F}.
\end{equation*}
Combining the two bounds together, we obtain
\begin{eqnarray*}
	\int_{\mathbb{R}^{P_{\eta}}} \bar{R} ( \vk{x} ) \, \bar{\varphi} ( \vk{x} ) \mathop{d \vk{x}}
	\leq
	C A \sum_{S \subset \{ 1, \ldots, P_{\eta} \}}
	\prod_{i \in F^c} ( \lambda_i + \varepsilon )^{-1}
	\exp \left(
	-\frac{1}{4 \delta} \, \left\| \vk{\lambda}_F + \varepsilon \, \vk{1}_F \right\|_2^2
	-\varepsilon \, \vk{1}_F^\top ( \vk{\lambda}_F + \varepsilon \, \vk{1}_F )
	\right)
	A_3 ( F ),
\end{eqnarray*}
where
\begin{equation*}
	A_3 ( F )
	\coloneqq
	\int_{\mathbb{R}^{|F|}} \exp \left( -\delta \left\|
	\vk{x}_F + \varepsilon \, \vk{1}_F - \frac{1}{2 \delta} \, \vk{\lambda}_{\vk{x}, F} ( \varepsilon )
	\right\|_2^2 \right)
	\left\| \vk{x}_F + \varepsilon \, \vk{1}_F \right\|_1^{\gamma}
	\mathop{d \vk{x}_F}
	+\frac{\varepsilon^{|S|}}{|S|!}
	< \infty
\end{equation*}
for all \( \delta > 0 \).

Hence, by
the dominated convergence theorem, it follows that as \( u \to \infty \)
\begin{equation}
	\label{Pickands_decomposition}
	\begin{aligned}
		M(u,L)
		& \sim u^{-P_\eta} \, \varphi_T(u\vk 1_\eta) \, e^{-\frac{L}{2}\vk \lambda^\top\vk\lambda}
			\int_{\R^{P_\eta}}\pk{\exists \, t\in[0,L]:~\vk B^*(t)>\vk x} \, e^{-\vk x^\top\vk\lambda}\td\vk x
		\\[7pt]
		&=
			\frac{
			\varphi_T(u\vk 1_\eta) \, e^{-\frac{L}{2}\vk \lambda^\top\vk\lambda}
			}{\prod_{i=1}^{P_\eta}\lambda_i} \, u^{-P_\eta}
			\int_{\R^{P_\eta}}\pk{\exists \, t\in[0,L] \colon \vk\lambda\vk B^*(t)>\vk x}e^{-\sum_{i=1}^{P_\eta}x_i} \, \td\vk x,
	\end{aligned}
\end{equation}
which, together with the following asymptotic formula (see~\cite[Lemma 4.4]{MR4569300})
\begin{equation}
	\label{gauss_assympotics}
	\pk{\vk W(T)>u\vk 1_{\eta}}
	\sim \frac{u^{-P_{\eta}}}{\prod_{i=1}^{P_\eta}\lambda_i}
	\, \varphi_T(u\vk 1_\eta),\qquad u\to\infty,
\end{equation}
provides us the claimed assertion.

\end{proof}

\begin{lem}\label{Pit_constant}
For $H(L)$ defined in~\eqref{H(L)} we have
\begin{align*}
	\mathcal{H}_{P_\eta,1/\mu_0(T)}=\limit{L}H(L)\in(0,\infty).
\end{align*}
\end{lem}

\begin{proof}[Proof of Lemma~\ref{Pit_constant}]
Applying the Fubini-Tonelli theorem yields
\begin{align*}
	H(L)
	&=	e^{-\frac{L}{2}\vk \lambda^\top\vk\lambda}
		\int_{\R^{P_\eta}}\pk{\exists~t\in[0,L]:~\vk\lambda\vk B^*_\eta(t)>\vk x}e^{\sum_{i=1}^{P_\eta} x_i}\td\vk x
	\\[7pt]
	&=
		\E{e^{-\vk \lambda^\top \vk B^*_\eta(L) -\frac{L}{2}\vk \lambda^\top\vk\lambda +\vk \lambda^\top \vk B^*_\eta(L)}
		\int_{\R^{P_\eta}}\mathbb{I} (\exists \, t\in[0,L] \colon \vk\lambda\vk B^*_\eta(t)>\vk x)
		e^{\sum_{i=1}^{P_\eta} x_i}\td\vk x }
	\\[7pt]
	&=
		\E{e^{-\vk \lambda^\top \vk B^*_\eta(L) -\frac{L}{2}\vk \lambda^\top\vk\lambda }
		\int_{\R^{P_\eta}}\mathbb{I} (\exists \, t\in[0,L] \colon
		\vk\lambda \left(\vk B^*_\eta(t)-\vk B^*_\eta(L)\right) >\vk x)e^{\sum_{i=1}^{P_\eta}x_i}\td\vk x }
	\\[7pt]
	&=
		\E{\int_{\R^{P_\eta}}\mathbb{I} (\exists \, t\in[0,L] \colon
		\vk\lambda\left(\vk B^*_\eta(t)-\vk B^*_\eta(L)\right) -\vk\lambda^2 (t - L)>\vk x)
		e^{\sum_{i=1}^{P_\eta}\vk x}\td\vk x }
	\\[7pt]
	&=
		\int_{\R^{P_\eta}}\pk{ \exists \, t\in[0,L] \colon \vk\lambda \vk B^{*}_\eta(t)-\vk \lambda^2 t >\vk x}
		e^{\sum_{i=1}^{P_\eta}x_i}\td\vk x,
\end{align*}
where the second to last step follows from~\cite[Lem
B.6]{hashorva2021shiftinvariant} and the last is a direct consequence of the
Brownian motion increments' properties. Hence, $H(L)$ is an increasing function.
Using the following inequality, which follows from \neprop{Korsh},
\begin{eqnarray*}
	(H(L)-\varepsilon)\pk{\vk B_\Gamma(T_1)>u\vk{1}_{i(T_1)}}
	\leq M(u,L)
	\leq\pk{\exists \, t\in[0,T]:~\vk B_\Gamma(t)>u\vk 1_{i(t)}}
	\leq 2^{P_\eta} \, \pk{\vk B_\Gamma(T_1)>u\vk{1}_{i(T_1)}}
\end{eqnarray*}
for any small positive $\varepsilon$ and large enough $u$, we obtain that
\begin{align*}
	H(L)\leq2^{P_\eta}+\varepsilon<\infty
\end{align*}
implies that $H(L)$ is bounded. In order to complete the proof it remains to
apply the monotone convergence theorem and notice that according to
\neprop{eigenvalues}
\begin{equation*}
	\vk\lambda=\frac{1}{\mu_0(t)}\vk 1_\eta.
\end{equation*}
 \end{proof}

\begin{lem}\label{m}
	For  $L>0$ and large enough $u$,
	\begin{equation*}
		\frac{m(u,L)}{\pk{\vk W(T)>u\vk 1_\eta}}\leq C_1e^{-C_2 L}
	\end{equation*}
	for some positive constants $C_1, \ C_2$, which do not depend on $u$ or $L$.
\end{lem}

\begin{proof}[Proof of Lemma~\ref{m}]
	Applying \neprop{Korsh}
	\begin{equation*}
		\frac{m ( u, L )}{\pk{ \vk{W} ( T ) > u \, \vk{1}_{\eta} }}
		\leq C \,
		\frac{
			\pk{ \vk{W} ( T - \delta_u ( L ) ) > u \, \vk{1}_{\eta} }
		}{
			\pk{ \vk{W} ( T ) > u \, \vk{1}_{\eta} }
		}
		= C \,
		\frac{
			\pk{ \vk{W} ( T ) > \vk{a} ( u, L, T )}
		}{
			\pk{ \vk{W} ( T ) > u \, \vk{1}_{\eta} }
		},
	\end{equation*}
	where $C$ defined in \neprop{Korsh} and
	\begin{equation*}
		\vk{a}( u, L, T )
		\coloneqq
		\sqrt{\frac{T}{T - \delta_u ( L )}}
		\Big( u \, \vk{1}_{\eta} + \vk{c} ( T - \delta_u ( L ) ) \Big)
		-\vk{c} T.
	\end{equation*}
	Hence, regarding \eqref{gauss_assympotics}, for any $\varepsilon>0$ and large enough $u$
	\begin{equation}\label{eq:4}
		\begin{aligned}
			\frac{m ( u, L )}{\pk{ \vk{W} ( T ) > u \, \vk{1}_{\eta} }}
			& \leq
				C \, ( 1 + \varepsilon ) \,
				\left(
				\frac{T}{T - \delta_u ( L )}
				\right)^{P_{\eta} / 2}
				\frac{
				\varphi_T \left( \vk{a} ( u, L, T ) \right)
				}{
				\varphi_T \left( u \vk{1}_{\eta} \right)
				}
			\\[7pt]
			& \leq
				C \, ( 1 + \varepsilon )^{1 + P_{\eta} / 2} \,
				\frac{
				\varphi_T \left( \vk{a} ( u, L, T ) \right)
				}{
				\varphi_T \left( u \vk{1}_{\eta} \right)
				}.
		\end{aligned}
	\end{equation}
	It remains to bound the quotient
	\( \varphi_T ( \vk{a} ( u, L, T ) ) / \varphi_T ( u \vk{1}_{\eta} ) \). To this end, let us
	first find the asymptotics of \( \vk{a} ( u, L, T ) \), as \( u \to \infty \)
	\begin{align*}
		\vk{a} ( u, L, T )
		& =
			u \vk{1}_{\eta} \left( 1 - \frac{\delta_u ( L )}{T} \right)^{-1/2}
			-\vk{c} T \left( 1 - \left( 1 - \frac{\delta_u ( L )}{T} \right)^{1/2} \right)
		\\[7pt]
		& \sim
			u \vk{1}_{\eta}
			\left( 1 +\frac{L}{2 T u} \right)
			+O ( L u^{-2} ).
	\end{align*}
	In this computation we used that \( \delta_u ( L ) = L u^{-2} \). Therefore,
	\begin{eqnarray*}
		\Big( \vk{a} ( u, L, T ) + \vk{c} T \Big)^\top \,
		\Sigma^{-1} ( T ) \,
		\Big( \vk{a} ( u, L, T ) + \vk{c} T \Big)
		&=&
		\left(
			u \vk{1}_{\eta}
			+\vk{c} T
		\right)^\top
		\Sigma^{-1} ( T )
		\left(
			u \vk{1}_{\eta}
			+\vk{c} T
		\right)\\
		&&+\frac{L}{T} \, \vk{1}_{\eta}^\top \Sigma^{-1} ( T ) \vk{1}_{\eta}
		+O ( L u^{-1} ),
	\end{eqnarray*}
	which yields
	\begin{equation*}
		\frac{\varphi_T ( \vk{a} ( u, L, T ) )}{\varphi_T ( u \vk{1}_{\eta} )}
		\leq
		C \, \exp \left( \frac{-L}{2T} \, \vk{1}_{\eta}^\top \Sigma^{-1} \vk{1}_{\eta} \right).
	\end{equation*}
	Combining this with~\eqref{eq:4}, we obtain the desired result.
\end{proof}

We can now proceed to the proof of \netheo{main}.


\begin{proof}[Proof of Theorem~\ref{main}]
	Using~\nelem{m}, we obtain that
	for any positive $L$ and large enough $u$
	\begin{align*}
		\frac{M(u,L)}{\pk{\vk W(T)>u\vk 1_{\eta}}}
		& \leq \frac{\pk{\exists \, t\in[0,T] \colon \vk W(t)>u\vk 1_{\eta}}}{\pk{ \vk W(T)>u\vk 1_{\eta}}}
		\\[7pt]
		& \leq \frac{M(u,L)}{\pk{\vk W(T)>u\vk 1_{\eta}}}
			+\frac{m(u,L)}{\pk{\vk W(T)>u\vk 1_{\eta}}},
	\end{align*}
	where the last term on the right is at most \( C_1 \exp ( -C_2 L ) \) with
	some positive constants \( C_1 \) and \( C_2 \).
	Hence, letting $u\to\infty$ and using \nelem{Pickands} we obtain that
	\begin{align*}
		H(L)
		& \leq \liminf_{u\to\infty}\frac{\pk{\exists t\in[0,T]:~\vk W(t)>u\vk 1_{\eta}}}{\pk{ \vk W(T)>u\vk 1_{\eta}}}
		\\[7pt]
		& \leq \limsup_{u\to\infty}\frac{\pk{\exists t\in[0,T]:~\vk W(t)>u\vk 1_{\eta}}}{\pk{ \vk W(T)>u\vk 1_{\eta}}}
			\leq H(L)+C_1e^{-C_2L}.
	\end{align*}
	Pushing further $L\to\infty$ and applying \nelem{Pit_constant} we obtain
	\begin{eqnarray*}
		\mathcal{H}_{P_\eta,1/\mu_0(T)}
		\leq \liminf_{u\to\infty}\frac{\pk{\exists t\in[0,T]:~\vk W(t)>u\vk 1_{\eta}}}{\pk{\vk W(T)>u\vk 1_{\eta}}}
		\leq \limsup_{u\to\infty}\frac{\pk{\exists t\in[0,T]:~\vk W(t)>u\vk 1_{\eta}}}{\pk{ \vk W(T)>u\vk 1_{\eta}}}
		\leq \mathcal{H}_{P_\eta,1/\mu_0(T)}.
	\end{eqnarray*}
	Hence, the claim follows.
\end{proof}

\begin{proof}[Proof of Corollary~\ref{random_N}]
Let us define $N^*_i=\operatorname{essinf} (N_i)$ and take
$\tilde{\vk N}_1,\tilde{\vk N}_2\in\N^\eta$ such that
$\tilde{\vk N}_1\geq \tilde{\vk N}_2$ and $\tilde{\vk N}_1\not= \tilde{\vk N}_2$.
Then we can show the following inequality
\begin{eqnarray}\label{compare_N}
	\pk{\exists \, t\in[0,T] \colon \vk W(t)>u\vk 1_\eta \ \middle| \  \vk N=\tilde{\vk N}_1}
	\leq\pk{\exists \, t\in[0,T] \colon \vk W(t)>u\vk 1_\eta \ \middle| \  \vk N=\tilde{\vk N}_2}.
\end{eqnarray}
To show~\eqref{compare_N} it is enough to consider
$\tilde{\vk N}_1,\tilde{\vk N}_2$ such that
\begin{align*}
	\tilde{\vk N}_1-\tilde{\vk N}_2=\vk e_j
\end{align*}
for some $j\in\{1,\ldots,\eta\}$, where $\vk e_j$ is the $j$-th basis vector in $\N^\eta$. It
is easy to see that the process $\vk W(t) \mid \vk N=\tilde{\vk N}_2$
may be obtained from $\vk W(t) \mid \vk N=\tilde{\vk N}_1$ by deleting some
branches. Thus, the inequality~\eqref{compare_N} follows.

Hence,
\begin{align*}
	&
	\pk{\vk N=\vk N^*} \,
	\pk{\exists \, t\in[0,T] \colon \vk W(t)>u\vk{1}_\eta \mid \vk N=\vk N^*}
	\\[7pt]
	& \hspace{50pt}
	\leq \pk{\exists \, t\in[0,T] \colon \vk W(t)>u\vk{1}_\eta}
	\\[7pt]
  & \hspace{50pt}
	\leq \pk{\vk N=\vk N^*} \,
	\pk{\exists \, t\in[0,T] \colon \vk W(t)>u\vk{1}_\eta \mid \vk N=\vk N^*}
	\\[7pt]
	& \hspace{70pt}
	+\sum_{i=1}^\eta
	\pk{\vk N\geq\vk N^*+\vk e_i} \,
	\pk{\exists \, t\in[0,T] \colon \vk W(t)>u\vk{1}_\eta \mid \vk N=\vk N^*+\vk e_i},
\end{align*}
and the claim follows from \netheo{main} and the fact that
\begin{equation*}
	\pk{\vk{W}(T)>u\vk{1}_\eta \mid \vk{N}=\vk N^* + \vk e_i}
	=o\left(\pk{\vk{W}(T)>u\vk{1}_\eta \mid \vk{N}=\vk N^*}\right).
\end{equation*}
To show the latter, we use \neprop{w(t)_asympt} as follows. Combining
\begin{align*}
	\left( P_\eta \mid \vk N=\vk N^*+\vk e_i \right)
	= \left( 1 + \frac{1}{N^*_i}\right)
	\left(P_\eta \mid \vk N=\vk N^*\right),
\end{align*}
with
\begin{align*}
	\left( \mu_0(T)\mid \vk N=\vk N^*+\vk e_i \right)
	& = \left( \mu_0(T)\mid \vk N=\vk N^* \right)
		+\sum_{l=1}^i (\tau_l-\tau_{l-1})\prod_{\substack{j=l\\j\not= i}}^\eta N_j\\
	& = \left( \mu_0(T)\mid \vk N=\vk N^* \right)
		+\frac{1}{N_i}\sum_{l=1}^i (\tau_l-\tau_{l-1})\prod_{j=l}^\eta N_j\\
	& \leq \left(1+\frac{1}{N_i}\right)
		\left( \mu_0(T) \mid \vk N=\vk N^* \right),
\end{align*}
we obtain
\begin{align*}
	\left( \frac{P_\eta}{\mu_0(T)} \, \middle| \, \vk N=\vk N^*+\vk e_i \right)
	\leq \left( \frac{P_\eta}{\mu_0(T)} \, \middle| \, \vk N=\vk N^* \right),
\end{align*}
and invoke \neprop{w(t)_asympt}.
\end{proof}

\begin{proof}[Proof of Corollary ~\ref{ruintime}]
Using  \kk{\netheo{main}}, we obtain
\begin{multline*}
\mathbb{P} \left\{ u^2 ( T - \mathcal{T} ( u ) ) \geq x \ \middle| \ \mathcal{T} ( u ) \leq T - \frac{y}{u^2} \right\}
= \frac{
\mathbb{P} \left\{ \mathcal{T} ( u ) \leq T - x u^{-2} \right\}
}{
\mathbb{P} \left\{ \mathcal{T} ( u ) \leq T - y u^{-2} \right\}
}
\\[7pt]
\sim
\frac{
	\mathcal{H}_{P_{\eta}, 1/\mu_0 ( T - y u^{-2} )} \,
	\mathbb{P} \left\{ \vk{W} ( T - x u^{-2} ) > u \, \vk{1}_{\eta} \right\}
}{
	\mathcal{H}_{P_{\eta}, 1/\mu_0 ( T - x u^{-2} )} \,
	\mathbb{P} \left\{ \vk{W} ( T - y u^{-2} ) > u \, \vk{1}_{\eta} \right\}
}
\sim
\frac{
\mathbb{P} \left\{ \vk{W} ( T - x u^{-2} ) > u \, \vk{1}_{\eta} \right\}
}{
\mathbb{P} \left\{ \vk{W} ( T - y u^{-2} ) > u \, \vk{1}_{\eta} \right\}
}.
\end{multline*}
Applying \neprop{w(t)_asympt}, we find that the latter is asymptotically
equivalent to
\begin{multline*}
 \frac{
		\mu_0^{P_{\eta} - 1/2} ( T - y u^{-2} )
		\prod_{v = 1}^{\eta} \mu_v^{( P_v - P_{v - 1} ) / 2} ( T - x u^{-2} )
	}{
		\mu_0^{P_{\eta} - 1/2} ( T - y u^{-2} )
		\prod_{v = 1}^{\eta} \mu_v^{( P_v - P_{v - 1} ) / 2} ( T - y u^{-2} )
	}
  \exp \left(
		-\frac{( u + c ( T - x u^{-2} ) )^2 P_{\eta}}{2 \mu_0 ( T - x u^{-2} )}
		+\frac{( u + c ( T - y u^{-2} ) )^2 P_{\eta}}{2 \mu_0 ( T - y u^{-2} )}
	\right).
\end{multline*}
Since the pre-exponential factor clearly tends to \( 1 \), it remains to compute
the asymptotics of the exponent. We have
\begin{multline*}
	-\frac{( u + c ( T - x u^{-2} ) )^2 P_{\eta}}{2 \mu_0 ( T - x u^{-2} )}
	+\frac{( u + c ( T - y u^{-2} ) )^2 P_{\eta}}{2 \mu_0 ( T - y u^{-2} )}
	\\[7pt]
	\sim
	-\frac{
		( u + c T )^2 P_{\eta}
	}{2 \mu_0 ( T - x u^{-2} ) \mu_0 ( T - y u^{-2} )}
	\Big[ \mu_0 ( T - y u^{-2} ) - \mu_0 ( T - x u^{-2} ) \Big].
\end{multline*}
It thus remains to note that
\begin{equation*}
	\mu_0 ( T - y u^{-2} ) - \mu_0 ( T - x u^{-2} )
	\sim  u^{-2} ( x - y )
  \quad \text{and} \quad
	\mu_0 ( T - x u^{-2} ), ~\mu_0 ( T - y n^{-2} ) \sim \mu_0 ( T )
\end{equation*}
to conclude the proof.
\end{proof}

\subsection{Proof of Theorem~\ref{forest_ruin}}
We note that
\kk{
\begin{eqnarray*}
	\pk{\exists \, i\in\{1,\ldots,M\}, \, t\in[0,T] \colon \vk W_i(t)>u\vk 1_{\eta_i}}
&\geq&
	\sum_{i=1}^M \pk{\exists \, t\in[0,T] \colon \vk W_i(t)>u\vk 1_{\eta_i}}\\
	&&
-\sum_{\substack{i,j=1\\i\not= j}}^M
	\pk{
		\begin{aligned}
			\exists \, t_1 \in [0, T] \colon \vk{W}_i ( t_1 ) > u \vk{1}_{\eta_i} \\
			\exists \, t_2 \in [0, T] \colon \vk{W}_j ( t_2 ) > u \vk{1}_{\eta_j}
		\end{aligned}
	}
\end{eqnarray*}
and
\begin{eqnarray*}
	\pk{\exists \, i\in\{1,\ldots,M\}, \, t\in[0,T] \colon \vk W_i(t)>u\vk 1_{\eta_i}}
		\leq \sum_{i=1}^M \pk{\exists \, t\in[0,T] \colon \vk W_i(t)>u\vk 1_{\eta_i}}.
\end{eqnarray*}
}
Using that $\vk W_i$ are independent,
\begin{eqnarray*}
\lefteqn{
	\pk{\exists \, t_1, \, t_2\in[0,T] \colon \vk W_i(t_1)>u\vk 1_{\eta_i}, \ \vk W_j(t_2)>u\vk 1_{\eta_j}}}\\
	&=& \pk{\exists \, t\in[0,T] \colon \vk W_i(t)>u\vk 1_{\eta_i}}
	\pk{\exists \, t\in[0,T] \colon \vk W_j(t)>u\vk 1_{\eta_j}}\\
	&&
	+o\left(
	\pk{\exists \, t\in[0,T] \colon \vk W_i(t)>u\vk 1_{\eta_i}}
	+\pk{\exists \, t\in[0,T] \colon \vk W_j(t)>u\vk 1_{\eta_j}}
	\right).
\end{eqnarray*}
We find
\begin{align*}
	\pk{\exists \, i\in\{1,\ldots,M\}, \, t\in[0,T] \colon \vk W_i(t)>u\vk 1_{\eta_i}}
	\sim \sum_{i=1}^M \pk{\exists \, t\in[0,T] \colon \vk W_i(t)>u\vk 1_{\eta_i}}.
\end{align*}
Finally, \nelem{tree_compare} implies that
\begin{align*}
	\sum_{i\not\in A} \pk{\exists \, t\in[0,T] \colon \vk W_i(t)>u\vk 1_{\eta_i}}
	= o\left( \sum_{i\in A}\pk{\exists \, t\in[0,T] \colon \vk W_i(t)>u\vk 1_{\eta_i}} \right),
\end{align*}
establishing the proof.
\QED

\section{Appendix}
\label{sec:auxiliary-proofs}
\begin{proof}[Proof of Proposition~\ref{b_decomp}]
Using mathematical induction, we shall show that for all
$i\in\{1,\ldots,N\}$ holds
\begin{equation}
	\label{b_tau}
	B_\gamma(\tau_i)=B_0^*(\tau_1)+\sum_{j=1}^{i-1}(\vk{B}^*_{j})_{(\gamma \bmod P_j)+1}(\tau_{j+1}-\tau_j).
\end{equation}
Since for $i=1$ the claim is clear, assuming it holds for $i=k$, we have for
$i=k+1$
\begin{align}\label{eq:3}
	B_{\gamma} ( \tau_{k + 1} )
	& = \big( \widetilde{\vk{B}}_{\Gamma} \big)_{( \gamma \bmod P_k ) + 1} ( \tau_{k + 1} )
	\\[7pt]
	& =
	\begin{pmatrix}
		\widetilde{\vk{B}}_{\Gamma} ( \tau_k ) \\
		\vdots \\
		\widetilde{\vk{B}}_{\Gamma} ( \tau_k )
	\end{pmatrix}_{( \gamma \bmod P_k ) + 1}
  + \big( \vk{B}^{*}_k \big)_{( \gamma \bmod P_k ) + 1} ( \tau_{k + 1} - \tau_k ).
\end{align}
For the first term we have
\begin{align*}
	\begin{pmatrix}
		\widetilde{\vk{B}}_{\Gamma} ( \tau_k ) \\
		\vdots \\
		\widetilde{\vk{B}}_{\Gamma} ( \tau_k )
	\end{pmatrix}_{( \gamma \bmod P_k ) + 1}
  & =
	\big( \widetilde{\vk{B}}_{\Gamma} ( \tau_k ) \big)_{\big( ( \gamma_k \bmod P_k ) \bmod P_{k - 1} \big) + 1}
	\\[7pt]
	& =
	\big( \widetilde{\vk{B}}_{\Gamma} ( \tau_k ) \big)_{( \gamma_k \bmod P_{k - 1} ) + 1}
	= B_{\gamma} ( \tau_k ).
\end{align*}
By the induction hypothesis,
\begin{equation}\label{eq:2}
	B_{\gamma} ( \tau_k )
	= B_0^{ * } ( \tau_1 )
	+\sum_{j = 1}^{k - 1} \big( \vk{B}_j^{ * } \big)_{( \gamma \bmod P_j ) + 1} ( \tau_{j + 1} - \tau_j ).
\end{equation}
Subsume the second term of~\eqref{eq:3} into the sum~\eqref{eq:2} and note that
this conclues the proof of this assertion for \( t = \tau_{k + 1} \). Similarly, if
\( t \in ( \tau_1, T ] \), then
\begin{equation*}
	B_{\gamma} ( t ) =
	\begin{pmatrix}
		\widetilde{\vk{B}}_{\Gamma} ( \tau_{i ( t )} ) \\
		\vdots \\
		\widetilde{\vk{B}}_{\Gamma} ( \tau_{i ( t )} )
	\end{pmatrix}_{( \gamma \bmod P_{i ( t )} ) + 1}
	+\big( \vk{B}^{ * }_k \big)_{( \gamma \bmod P_{i ( t )} ) + 1} ( t - \tau_{i ( t )} ),
\end{equation*}
and therefore
\begin{equation*}
	B_{\gamma} ( t )
	=
	B_0^{ * } ( \tau_1 )
	+\sum_{j = 1}^{i ( t ) - 1} \big( \vk{B}^{ * }_j \big)_{( \gamma \bmod P_j ) + 1} ( \tau_{j + 1} - \tau_j )
	+\big( \vk{B}^{ * }_{i ( t )} \big)_{( \gamma \bmod P_{i ( t )} ) + 1} ( t - \tau_{i ( t )} ).
\end{equation*}
Thus, the claim~\eqref{b_sum} follows, justifying that $B_\gamma$ is a Brownian
motion.
\bigskip

Consider now the assertion~\eqref{cov_formula}. Let $t_1\leq\tau_{\kappa(\gamma_1,\gamma_2)}$. Using that
$B_{\gamma_1}(t)=B_{\gamma_2}(t)$ for any $t\leq\tau_{\kappa(\gamma_1,\gamma_2)}$, we have
\begin{align*}
	B_{\gamma_1}(t_1)=B_{\gamma_2}(t_1).
\end{align*}
Since $B_{\gamma_2}$ is a Brownian motion, we obtain
\begin{align*}
	\Cov(B_{\gamma_1}(t_1),B_{\gamma_2}(t_2))
	= \Cov(B_{\gamma_2}(t_1),B_{\gamma_2}(t_2))
	= \min\{t_1,t_2\}
	= \min\{t_1,t_2,\tau_{\kappa(\gamma_1,\gamma_2)}\}.
\end{align*}
The same holds in the case $t_2\leq\tau_{\kappa(\gamma_1,\gamma_2)}$. In the case
$\tau_{\kappa(\gamma_1,\gamma_2)}<t_1,t_2$, using~\eqref{b_sum}, we can see that
\begin{multline*}
	B_{\gamma_1}(t_1)
	=
	B_0^*(\tau_1)
	+\sum_{j=1}^{i(t_1)-1}
	\big( \vk{B}_j^{*} \big)_{(\gamma_1 \bmod P_{j})+1}(\tau_{j+1}-\tau_j)
	+\big( \vk{B}^*_{i(t_1)} \big)_{(\gamma_1 \bmod P_{i(t_1)})+1}(t_1-\tau_{i(t_1)}).
\end{multline*}
Next, we split the middle sum at the branches' separation point
\( \kappa ( \gamma_1, \gamma_2 ) \) and use the fact that
\begin{equation*}
	B_0^{*} ( \tau_1 )
	+\sum_{j = 1}^{\kappa ( \gamma_1, \gamma_2 ) - 1} \big( \vk{B}_j^* \big)_{( \gamma_1 \bmod P_j ) + 1} ( \tau_{j + 1} - \tau_j )
	=
	B_{\gamma_1} ( \tau_{\kappa ( \gamma_1, \gamma_2 )} )
\end{equation*}
to obtain
\begin{multline*}
	B_{\gamma_1}(t_1)
	B_{\gamma_1}(\tau_{\kappa(\gamma_1,\gamma_2)})
	+\sum_{j=\kappa(\gamma_1,\gamma_2)}^{i(t_1)-1}
	\big( \vk{B}^*_j \big)_{(\gamma_1 \bmod P_{j})+1}(\tau_{j+1}-\tau_j)
	+\big( \vk{B}^*_{i(t_1)} \big)_{(\gamma_1 \bmod P_{i(t_1)})+1}(t_1-\tau_{i(t_1)}).
\end{multline*}
By the same reason,
\begin{multline*}
	B_{\gamma_2}(t_2)
	= B_{\gamma_2}(\tau_{\kappa(\gamma_1,\gamma_2)})
	+\sum_{j=\kappa(\gamma_1,\gamma_2)}^{i(t_2)-1}
	\big( \vk{B}^*_j \big)_{(\gamma_2 \bmod P_{j})+1}(\tau_{j+1}-\tau_j)
	+\big( \vk{B}^*_{i(t_2)} \big)_{(\gamma_2 \bmod P_{i(t_2)})+1}(t_2-\tau_{i(t_2)}).
\end{multline*}
We have used the fact that for all indices $j\geq\kappa(\gamma_1,\gamma_2)$ holds
$(\gamma_1 \bmod P_j)\not= (\gamma_2 \bmod P_j)$, so the sums
\begin{align*}
	\sum_{j=\kappa(\gamma_1,\gamma_2)}^{i(t_1)-1}
	\big( \vk{B}^*_j \big)_{(\gamma_1 \bmod P_j)+1}(\tau_{j+1}-\tau_j)
	+\big( \vk{B}^*_{i(t_1)} \big)_{(\gamma_1 \bmod P_{i(t_1)})+1}(t_1-\tau_{i(t_1)})
\end{align*}
and
\begin{align*}
	\sum_{j=\kappa(\gamma_1,\gamma_2)}^{i(t_2)-1}
	\big( \vk{B}^*_j \big)_{(\gamma_2 \bmod P_j)+1}(\tau_{j+1}-\tau_j)
	+\big( \vk{B}^*_{i(t_2)} \big)_{(\gamma_2 \bmod P_{i(t_2)})+1}(t_2-\tau_{i(t_2)})
\end{align*}
are independent. Additionally, each of them is independent of
$B_{\gamma_1}(\tau_{\kappa(\gamma_1,\gamma_2)})$, which is equal to $B_{\gamma_2}(\tau_{\kappa(\gamma_1,\gamma_2)})$. Hence,
\begin{eqnarray*}
	\Cov(B_{\gamma_1}(t_1),B_{\gamma_2}(t_2))
	= \Cov(B_{\gamma_1}(\tau_{\kappa(\gamma_1,\gamma_2)}),B_{\gamma_2}(\tau_{\kappa(\gamma_1,\gamma_2)}))
	= \Var(B_{\gamma_1}(\tau_{\kappa(\gamma_1,\gamma_2)}))
	= \tau_{\kappa(\gamma_1,\gamma_2)}.
\end{eqnarray*}
This establishes~\eqref{cov_formula}.
\end{proof}

\begin{proof}[Proof of Proposition~\ref{eigenvalues}]
	We are going to prove the claim of the theorem by induction in
	$t\in\{\tau_0,\ldots,\tau_{\eta}\}$. The case $t=\tau_0$ is clear. Assume that the claim is true
	for $t=\tau_k$. We claim that the eigenvalues of
\begin{equation*}
	\begin{pmatrix}\Sigma(\tau_k) &\ldots & \Sigma (\tau_k) \\ \vdots & & \vdots \\ \Sigma(\tau_k) & \ldots & \Sigma(\tau_k)\end{pmatrix}
\end{equation*}
are the eigenvalues of $\Sigma(\tau_k)$ multiplied by $N_k$. Moreover, they are of at
least the same multiplicity. Indeed, for any eigenvalue $\mu$ and a corresponding
eigenvector $\vk v\in\R^{P_{k-1}}$ of $\Sigma(\tau_k)$ we can construct a vector
$\vk v^*=(\vk v^\top,\ldots,\vk v^\top)^\top\in\R^{P_k}$, which satisfies
\begin{equation*}
	\begin{pmatrix}
		\Sigma(\tau_k) &\ldots & \Sigma (\tau_k) \\ \vdots & & \vdots \\ \Sigma(\tau_k) & \ldots & \Sigma(\tau_k)
	\end{pmatrix}
	\begin{pmatrix}\vk v\\ \vdots \\ \vk v\end{pmatrix}
	=
	\begin{pmatrix}
		N_k\Sigma(\tau_k)\vk v \\ \vdots \\ N_k\Sigma(\tau_k)\vk v
	\end{pmatrix}
	= \begin{pmatrix}N_k\mu\vk v \\ \vdots \\ N_k\mu\vk v\end{pmatrix}
	=N_k \, \mu \, \vk v^*.
\end{equation*}
In particular, it means that as $\vk 1_{k-1}$ is an eigenvector of $\Sigma(\tau_k)$ with
the eigenvalue $\mu_0(\tau_k)$, then $\vk 1_{k}$ is an eigenvector of the matrix
mentioned above with the corresponding eigenvalue $N_k \, \mu_0(\tau_k)$. Since
\begin{align*}
	\rank
	\begin{pmatrix}
		\Sigma(\tau_k) &\ldots & \Sigma (\tau_k) \\ \vdots & & \vdots \\ \Sigma(\tau_k) & \ldots & \Sigma(\tau_k)
	\end{pmatrix}
	= \rank \Sigma(\tau_k)
	= P_{k-1},
\end{align*}
\( 0 \) is an eigenvalue of the matrix mentioned above of multiplicity
$P_k-P_{k-1}$. Using Proposition~\ref{Sigma_rec} for any eigenvalue $\mu_v(\tau_k)$
of $\Sigma(\tau_k)$ we can find an eigenvalue $\mu_v(\tau_{k+1})$ of $\Sigma(\tau_{k+1})$ as follows
\begin{align*}
	\mu_v(\tau_{k+1})&=N_k\mu_v(\tau_k)+(\tau_{k+1}-\tau_k)\\
	&=(\tau_{k+1}-\tau_k)+N_k\left((\tau_{k}-\tau_{k-1})+\sum_{l=v+1}^{k-1}(\tau_l-\tau_{l-1})\prod_{j=l}^{k-1}N_j\right)\\
	&=(\tau_{k+1}-\tau_k)+N_k(\tau_k-\tau_{k-1})+\sum_{l=v+1}^{k-1}(\tau_l-\tau_{l-1})\prod_{j=l}^{k}N_j\\
	&=(\tau_{k+1}-\tau_k)+\sum_{l=v+1}^{k}(\tau_l-\tau_{l-1})\prod_{j=l}^{k}N_j.
\end{align*}
The multiplicity of $\mu_v(\tau_{k+1})$ equals to the multiplicity of $\mu_v(\tau_k)$,
i.e. $P_{v}-P_{v-1}$. In particular, for $i=0$ we have that $\mu_0(\tau_{k+1})$ has
exactly one eigenvector $\vk 1_k$. Additionally, $\Sigma(\tau_{k+1})$ has the eigenvalue
$\tau_{k+1}-\tau_k$ with multiplicity $P_k-P_{k-1}$. Hence, the claim follows for the
matrix $\Sigma(\tau_{k+1})$. Finally, using Proposition~\ref{Sigma_rec} we can find
eigenvalues and eigenvectors of $\Sigma(t)$ if we know them for $\Sigma(\tau_{i(t)})$.
\end{proof}

\begin{proof}[Proof of Proposition~\ref{w(t)_asympt}]
Using that $\vk\lambda=\Sigma^{-1}\vk 1_\eta$ and \neprop{eigenvalues}
we obtain
\begin{align*}
	\vk\lambda = \frac{1}{\mu_0(T)} \, \vk 1_\eta,
	\qquad \vk 1_\eta^\top\Sigma^{-1}(T)\vk 1_\eta =\frac{P_\eta}{\mu_0(T)}.
\end{align*}
Additionally, using \neprop{eigenvalues} we know that
\begin{align*}
	\abs{\Sigma(T)}=\mu_0(T)\prod_{v=1}^{\eta}\mu_v^{P_v-P_{v-1}}(T).
\end{align*}
Hence, the claim follows from~\eqref{gauss_assympotics}:
\begin{eqnarray*}
	\pk{\vk B_\Gamma(T)-cT\vk 1_{\eta}>u\vk 1_{\eta}}
	&\sim&
	\frac{1}{\prod_{i=1}^{P_\eta}\lambda_i}u^{-P_\eta}\varphi_T(u\vk 1_\eta)
	\\
	&=&
	\frac{u^{-P_{\eta}}}{\prod_{i=1}^{P_\eta}\lambda_i}
	\frac{1}{(2\pi)^{P_\eta/2}\abs{\Sigma(T)}^{1/2}}
	\exp\left(-\frac{1}{2}(u+cT)^2\vk 1_\eta^\top\Sigma^{-1}(T)\vk 1_\eta\right).
\end{eqnarray*}
Recall that $\varphi(T)$ stands for the pdf of $\vk B_\Gamma(T)-cT\vk 1_\eta$.
\end{proof}

\begin{proof}[Proof of Lemma~\ref{quadratic_solution_prop}]
	Note that any element $\gamma\in\Gamma$ has unique representation of the following form
\begin{align*}
	\gamma=\sum_{i=1}^{\eta} a_iP_{i-1},
\end{align*}
where $a_i\in\{0,1,\ldots,N_i-1\}$. Hence, the map
\begin{align*}
	\gamma \mapsto ( a_1(\gamma),a_2(\gamma),\ldots,a_\eta(\gamma) )
\end{align*}
is a bijection. Let us introduce the following class of permutations on \( \Gamma \):
\begin{equation*}
	\pi_{j,b,c}(\gamma)=\gamma^\prime,
	\qquad
  i \in \{1,2,\ldots,\eta\}, \quad
	b, \, c \in\{0,1,\ldots,N_i-1\},
\end{equation*}
where for any $i\in\{1,2,\ldots, \eta\}$
\begin{equation*}
	a_i(\gamma^\prime)
	= \begin{cases}
		a_i(\gamma), & \text{if} \, i\not=j,\\
		a_j(\gamma), & \text {if} \, i=j, \ a_j(\gamma)\not=c \ \text{and} \ a_j(\gamma) \not= b,\\
		b, & \text{if} \, i=j \ \text{and} \ a_j(\gamma)=c,\\
		c, & \text{if} \, i=j \ \text{and} \ a_j(\gamma)=b.\\
	\end{cases}
\end{equation*}
Next, define a new Gaussian random vector
$\tilde{\vk X}(j,b,c)=(\tilde{X}_0,\ldots,\tilde{X}_{P_\eta-1})\in\R^{P_\eta}$
\begin{equation*}
	\tilde X_i=B_{\pi_{j,b,c}(i)}(T)
\end{equation*}
and denote its covariance function by $\Sigma(j,b,c)$. Consider a new quadratic
programming problem (see Lemma~\ref{lem:quadratic_programming})
\begin{align*}
	\Pi_{\Sigma(j,b,c)}(\vk 1_\eta).
\end{align*}
Define its corresponding sets by $I(j,b,c)$ and $J(j,b,c)$. Using that our
vector $\tilde{\vk X}(j,b,c)$ is a permutation of $\vk B_\Gamma(T)$,
\begin{align*}
	I(j,b,c)=\pi_{j,b,c}(I).
\end{align*}
On the other hand, from the definition of $\pi_{j,b,c}$ we have that for any
$i\in\{1,\ldots,\eta\}$
\begin{align*}
	a_i(\gamma_1)=a_i(\gamma_2) \iff  a_i(\pi_{j,b,c}(\gamma_1))=a_i(\pi_{j,b,c}(\gamma_2)).
\end{align*}
Using the formula for $\gamma$ in terms of $a_i(\gamma)$ we obtain that
\begin{align*}
	\gamma_1=\gamma_2\bmod P_i \iff \forall\, l\in\{1,2,\ldots, i\}~a_l(\gamma_1)=a_l(\gamma_2).
\end{align*}
Combining these two facts, we find
\begin{align*}
	\gamma_1=\gamma_2\bmod P_i \iff \pi_{j,b,c}(\gamma_1)=\pi_{j,b,c}(\gamma_2)\bmod P_i,
\end{align*}
which implies that
\begin{align*}
	\kappa(\gamma_1,\gamma_2)=\kappa(\pi_{j,b,c}(\gamma_1),\pi_{j,b,c}(\gamma_2)).
\end{align*}
Consequently, we have
\begin{eqnarray*}
	\Cov(\tilde X_{\gamma_1},\tilde X_{\gamma_2})
	  &=& \Cov(B_{\pi_{j,b,c}(\gamma_1)}(T),B_{\pi_{j,b,c}(\gamma_2)}(T))
	\\
		&=& \min \{ T,\tau_{\kappa(\pi_{j,b,c}(\gamma_1),\pi_{j,b,c}(\gamma_2))} \}\\
	  &=& \min \{ T,\tau_{\kappa(\gamma_1,\gamma_2)} \}
		= \Cov(B_{\gamma_1}(T),B_{\gamma_2}(T)).
\end{eqnarray*}
Hence $\Sigma(j,b,c)=\Sigma(T)$, and
\begin{equation*}
	I(j,b,c)=I.
\end{equation*}
It means that for any $j\in\{1,\ldots,\eta\}$, $b,c\in\{0,\ldots,N_j\}$
\begin{align*}
	\pi_{j,b,c}(I)=I.
\end{align*}
Finally, let us notice that for any $\gamma_1, \, \gamma_2\in\Gamma$ we can find a sequence of the
permutations of the class mentioned above which send $\gamma_1$ into $\gamma_2$:
\begin{equation*}
	\gamma_2
	=\pi_{1,a_1(\gamma_1),a_1(\gamma_2)}(\pi_{2,a_2(\gamma_1),a_2(\gamma_2)}(\cdots \pi_{\eta,a_\eta(\gamma_1),a_\eta(\gamma_2)}(\gamma_1)\cdots )).
\end{equation*}
This implies that set $I$ can either be empty, or contain all elements of $\Gamma$.
Since by \nelem{lem:quadratic_programming} it cannot be empty, the claim
follows.
\end{proof}

Below, we cite the result from~\cite[Lemma 4]{DHJT15}.

\begin{lem}\label{lem:quadratic_programming}  Let $d \geq 2$ and
	$\Sigma$ a $d \times d$ symmetric positive definite matrix with inverse $\Sigma^{-1}$. If
	$\vk{a}\in\R^d \setminus (-\infty, 0]^d $, then the quadratic programming problem
	$\Pi_{\Sigma}(\vk b)$ defined in~\eqref{eq:QP} has a unique solution $\tilde{\vk a}$
	and there exists a unique non-empty index set $I\subset \{1 \ldot d\}$ with
	$|I| \le d$ elements such that
	\begin{itemize}
		\item[(i)] $\tilde{\vk{a}}_{I} = \vk{a}_{I} \not=\vk{0}_I$;
		\item[(ii)] $\tilde{\vk{a}}_{J} = \Sigma_{IJ}^{-1} \Sigma_{II}^{-1} \vk{a}_{I}\ge \vk{a}_{J}$, and $\Sigma_{II}^{-1} \vk{a}_{I}>\vk{0}_I$;
		\item[(iii)] $\min_{\vk{x} \ge \vk{a}}\vk{x}^\top \Sigma^{-1}\vk{x}
		= \tilde{\vk{a}} ^\top \Sigma^{-1} \tilde{\vk{a}}
		= \vk{a} ^\top \Sigma^{-1} \tilde{\vk{a}}
		= \vk{a}_{I}^\top \Sigma_{II}^{-1}\vk{a}_{I}>0$,
	\end{itemize}
	with $\vk{\lambda}= \Sigma^{-1} \tilde{\vk{a}}$ satisfying $\vk{\lambda}_I= \Sigma_{II}^{-1} \vk{a}_I> \vk 0_I$, and $\vk{\lambda}_J= \vk{0}_J$.
\end{lem}

\section*{Acknowledgements}
Financial support from the Swiss National Science Foundation Grant 1200021-196888 is kindly
acknowledged.
K. D\c{e}bicki
was partially supported by
NCN Grant No 2018/31/B/ST1/00370
(2019-2024).
N. Kriukov has received funding from the European Union's Horizon 2020 research and innovation programme under the Marie Skłodowska-Curie grant agreement Grant Agreement No 101034253.
A substantial part of the project was done when N. Kriukov was with Department of Actuarial Science,
University of Lausanne, Switzerland.

\bibliographystyle{ieeetr}
\bibliography{EEEA_Pavel}
\end{document}